\documentclass[12pt]{amsart}
\usepackage[toc,page]{appendix}
\usepackage{amsmath,amsfonts,amssymb,amsthm,mathtools}
\usepackage[vcentermath]{youngtab}
\usepackage[all]{xy}
\usepackage{csquotes}
\usepackage[usenames,dvipsnames]{xcolor}
\usepackage{pstricks}
\usepackage{filecontents}
\usepackage{graphicx}
\usepackage{cite}
\usepackage{setspace}
\usepackage{tabu}
\usepackage{tikz}
\usetikzlibrary{decorations.markings}
\usetikzlibrary{patterns}
\usepackage[margin=1in]{geometry}
\usepackage{siunitx} %
\usepackage{hyperref}
\hypersetup{colorlinks,  citecolor=blue,  linkcolor=BrickRed, urlcolor=black}

\usetikzlibrary{chains}

\usepackage{indentfirst}
\usepackage{comment}
\usepackage{todonotes}
\DeclareFontFamily{OT1}{rsfs}{}
\DeclareFontShape{OT1}{rsfs}{n}{it}{<-> rsfs10}{}
\DeclareMathAlphabet{\mathscr}{OT1}{rsfs}{n}{it}

\DeclareMathOperator{\dist}{dist}

\DeclareMathOperator{\Reg}{Reg}
\DeclareMathOperator{\Irr}{Irr}
\DeclareMathOperator{\Tr}{T}
\DeclareMathOperator{\Core}{Core}
\DeclareMathOperator{\Quot}{Quot}
\DeclareMathOperator{\res}{res}
\DeclareMathOperator{\reg}{reg}
\DeclareMathOperator{\sgn}{sgn}
\DeclareMathOperator{\Mb}{M_b}
\DeclareMathOperator{\Mp}{M_p}
\DeclareMathOperator{\Mtwo}{M_2}

\DeclareMathOperator{\MnT}{M_nT}
\DeclareMathOperator{\MfourT}{M_4T}

\DeclareMathOperator{\MkkT}{M_{b_{k}}T}

\DeclareMathOperator{\MkT}{M_{b_{k+1}}T}
\DeclareMathOperator{\MbT}{M_bT}
\DeclareMathOperator{\MdT}{M_dT}
\DeclareMathOperator{\Cr}{colreg}
\DeclareMathOperator{\Crab}{colreg_{a,b}}
\DeclareMathOperator{\Crbab}{colreg_{b-a,b}}

\DeclareMathOperator{\Crabk}{colreg_{a_{k+1},b_{k+1}}}
\DeclareMathOperator{\Slope}{Slope}
\DeclareMathOperator{\X}{X}
\DeclareMathOperator{\I}{I}
\DeclareMathOperator{\J}{J}
\DeclareMathOperator{\Xb}{X_b}

\makeatletter
\renewcommand{\boxed}[1]{\text{\fboxsep=.2em\fbox{\m@th$\displaystyle#1$}}}
\makeatother

\newcommand{\con}{\textrm{-}}

\theoremstyle{plain}
\newtheorem{Theorem}{Theorem}[section]
\newtheorem{Proposition}[Theorem]{Proposition}
\newtheorem{Lemma}[Theorem]{Lemma}
\newtheorem{Corollary}[Theorem]{Corollary}
\newtheorem{Conjecture}[Theorem]{Conjecture}
\newtheorem{Example}[Theorem]{Example}

\theoremstyle{definition}
\newtheorem{Definition}[Theorem]{Definition}

\theoremstyle{remark}
\newtheorem{Remark}[Theorem]{Remark}
\newtheorem{Claim}[Theorem]{Claim}

\usepackage[foot]{amsaddr}

\makeatletter
\renewcommand{\email}[2][]{%
  \ifx\emails\@empty\relax\else{\g@addto@macro\emails{,\space}}\fi%
  \@ifnotempty{#1}{\g@addto@macro\emails{\textrm{(#1)}\space}}%
  \g@addto@macro\emails{#2}%
}
\makeatother

\title{The title}%

\author{Panagiotis Dimakis}
\address[P. Dimakis]{Department of Mathematics\\ Stanford University\\ Stanford, CA, 94305, USA}
\email[P. Dimakis]{pdimakis12345@gmail.com}

\author{Guangyi Yue}
\address[G. Yue]{Department of Mathematics\\ Massachusetts Institute of Technology\\ Cambridge, MA, 02139, USA}
\email[G. Yue]{gyyue@mit.edu}

\begin{document}

\setcounter{tocdepth}{4}

\title{Combinatorial wall-crossing and the Mullineux Involution}
%\author{Panagiotis Dimakis\thanks{pdimakis12345@gmail.com\\Department of Mathematics, Stanford University, Stanford, CA, 94305, USA} \and Guangyi Yue\thanks{gyyue@mit.edu, Department of Mathematics, Massachusetts Institute of Technology, Cambridge, MA, 02139, USA}}

%\author{Panagiotis Dimakis}
%\email{pdimakis12345@gmail.com}
%\address{Department of Mathematics\\ Stanford University\\ Stanford, CA, 94305, USA}

%\author{Guangyi Yue}
%\email{gyyue@mit.edu}
%\address{Department of Mathematics\\ Massachusetts Institute of Technology\\ Cambridge, MA, 02139, USA}

%\date{\today}
\maketitle
\begin{abstract}
In this paper, we define the combinatorial wall-crossing transformation and the generalized column regularization on partitions and prove that a certain composition of these two transformations has the same effect on the one-row partition $(n)$. As corollaries we explicitly describe the quotients of the partitions which arise in this process. We also prove that the one-row partition is the unique partition that stays regular at any step of the wall-crossing transformation.

\end{abstract}

\tableofcontents

\section{Introduction}
Wall-crossing functors appear in the context of infinite-dimensional representations of complex semisimple Lie algebras, and Beilinson and Ginzburg studied its relation with translation functors in \cite{beilinson1999wall}. More recently, wall-crossing functors have appeared in the study of quantized symplectic resolutions of singularities as perverse equivalences between different categories of modules, for more details one can look at \cite{anno2011stability,bezrukavnikov2017dimension,losev2015supports}. These perverse equivalences induce bijections between irreducible objects in the corresponding derived categories, which are referred to as the combinatorial wall-crossing. In the classical case of Lie algebra representations they are related to the cactus group actions \cite{halacheva2017crystals,losev2015cacti}. It is called combinatorial wall-crossing because in case of rational Cherednik algebras of type $A$, the derived categories are parametrized by rational numbers, and the bijection among two categories parametrized by consecutive rational numbers (with denominator bounded above) is like crossing a wall.  Our work is motivated by combinatorial wall-crossing for representations of rational Cherednik algebras in large positive characteristic, where the combinatorial wall-crossing is given by the extended Mullineux involution due to Losev \cite{losev2015supports}. This collection of permutations on the set of partitions is our main object of study.

%over Cherednik algebras at different parameters in positive characteristic, for more details one can look at \cite{anno2011stability,bezrukavnikov2017dimension,losev2015supports}. These perverse equivalences induce bijections between irreducible objects in the corresponding derived categories. Combinatorially these bijections are given by the extended Mullineux transpose transformation and is called the combinatorial wall-crossing. Additionally, this perverse equivalence also leads to the study of Cactus group actions in \cite{halacheva2017crystals,losev2015cacti}, where the action of a cactus group of a Weyl group on the Weyl group itself is realized by wall-crossing functors.

Based on Kleshchev's work in \cite{kleshchev1996branching,ford1997proof}, the irreducible $p\con$modular representations of the symmetric group $\mathcal{S}_n$ are labeled by the $p\con$regular partitions of $n$; we denote the irreducible representation corresponding to the $p\con$regular partition $\lambda$ by $\rho_\lambda$.  
\begin{Definition}
\label{RepMull}
The Mullineux involution $\Mp$ is the involution on the set of $p\con$regular partitions satisfying
$$\rho_{\lambda^{\Mp}}=\rho_\lambda\otimes\sgn$$ where $\sgn$ is the sign representation. 
\end{Definition}
There are a few combinatorial ways to define $\Mp$ in \cite{kleshchev1996branching,ford1997proof}, where $p$ is not necessarily prime, and this is the foundation of our investigation.

In this paper, we study the behavior of one-row partition $(n)$ under composition of a series of wall-crossing transformations intensively. A strong monotonicity property motivates a generalized version of column regularization on partitions, which was originally defined in \cite{james16representation}. The relationship of Mullineux map and the original column regularization was studied by Walker, Bessenrodt, Olsson and Xu in \cite{walker1996horizontal,walker1994modular,bessenrodt1999properties,bessenrodt1998residue}. We generalize the definition of column regularization to two co-prime parameters, which can be understood as a rational number in the unit interval. This construction leads to the main result of this paper given in Theorem \ref{main}, which is that the combinatorial wall-crossing and a certain composition of generalized column regularization procedures have the same effect on the one-row partition. 

The most important consequence of this result, stated as Theorem \ref{monotone}, is that this one-row partition case is the only case where monotonicity holds at each step of the composition of the transformations. We were kindly informed by Losev that he has an alternative proof of Theorem \ref{monotone} using Heisenberg actions and perverse equivalences while our method is purely combinatorial. Also Theorem \ref{monotone} answers a question by Bezrukavnikov which is motivated by potential applications to the study of nabla operators and Haiman's $n!$ conjecture in \cite{haiman2002combinatorics}.

The paper is organized as follows. Section \ref{secpre} is an overview of preliminaries. The combinatorial wall-crossing transformation is defined in Section \ref{mainsection} and is followed by the main theorem. Then in Section \ref{secquotient} a description of the quotients of the series of partitions that arise when we apply the wall-crossing transformation to the one-row partition is presented and the property of uniqueness of monotonicity is proved in Section \ref{secunique}. We end with an explicit demonstration of the wall-crossing transformation to every partition of 5, and a general conjecture given by Bezrukavninov in the Appendix \ref{appendixa}.

{\bf Acknowledgements.} The authors would like to thank Roman Bezrukavnikov for suggesting this project to us and continuous discussions and help throughout the whole process. Also, the authors are grateful to Ivan Losev and to Galyna Dobrovolska for many discussions and to Seth Shelley-Abrahamson for useful revision suggestions.

\section{Preliminaries}
%partitions, leg, arm, hook, good, co-good, MT, colreg

\label{secpre}
A \textit{partition} $\lambda$ of $n\in\mathbb{N}$ is a finite tuple of weakly decreasing positive integers $\lambda=(\lambda_1,...,\lambda_k)$ where $\lambda_1\geq...\geq\lambda_k>0$ and $|\lambda|\coloneqq\sum_{i=1}^k\lambda_i=n$. The exponential version of a partition is $\lambda=(\lambda_1^{s_1},...,\lambda_k^{s_k})$, where the superscript $s_i$ indicates the number of repetitions of the part $\lambda_i$ and $\lambda_1>...>\lambda_k$. Denote all partitions by $\mathcal{P}$ and partitions of $n$ by $\mathcal{P}_n$. The \textit{Young diagram} corresponding to a given partition $\lambda$ is the set of unit boxes specified as follows. Fix the x-axis pointing to the south and the y-axis pointing to the east. Then the coordinates of the  southeast vertices of the boxes of the diagram are given by: $$(i,j)\in\{\mathbb{N}\times \mathbb{N}\mid 1\le i, 1\le j \le \lambda_i\}.$$
We also label by $(i,j)$ the box whose southeast vertex has coordinates $(i,j)$.
The transpose $\lambda^{\Tr}$ of a Young diagram $\lambda$ is given by: $$\{(i,j)\in\mathbb{N}\times \mathbb{N}\mid 1\le j, 1\le i \le \lambda_j\}.$$ 
Given a box $(i,j)\in \lambda$, the \textit{arm} $a_{ij}=a_{ij}(\lambda)$ is the set of boxes $(i,j')\in\lambda$ with $j<j'$. We use $a_{ij}$ to denote either the above set or the number of elements of the above set interchangeably. 
Similarly, the \textit{leg} $l_{ij}=l_{ij}(\lambda)$ is the set of boxes $(i',j)\in\lambda$ with $i<i'$. We use $l_{ij}$ to denote either the above set or the number of elements of the above set interchangeably as well. Finally the \textit{hook} $H_{ij}=H_{ij}(\lambda)$ is the union of sets $(i,j)\cup a_{ij}\cup l_{ij}$. The number of elements of the hook is also denoted by $H_{ij}$ and is equal to $1+a_{ij}+l_{ij}$.

Let $\lambda\in\mathcal{P}_n$. A box $A\in\lambda$ is called a \textit{removable} box of $\lambda$ if $\lambda\setminus A\in\mathcal{P}_{n-1}$. A box $B\notin\lambda$ is called an \textit{addable} box of $\lambda$, if $\lambda\cup B\in\mathcal{P}_{n+1}$. The \textit{rim} of $\lambda$ consists of the boxes $(i,j)\in\lambda$ such that $(i+1,j+1)\notin\lambda$. The \textit{boundary} of $\lambda$ is defined to be the rim of  $\tilde{\lambda}$ where $$\tilde{\lambda}=\lambda\bigcup_{(i,j)\mbox{ addable to }\lambda}(i,j).$$

%hook correspond to a rim segment
A skew shape $\lambda/\mu$, where $\mu\subset\lambda$, is the collection of boxes in $\lambda$ but not in $\mu$. If $\lambda/\mu$ does not contain any $2\times 2$ squares, then it is called a \textit{ribbon}. Note that every $(i,j)\in\lambda$ corresponds to a ribbon of size $H_{ij}$ containing in the rim of $\lambda$.

Fix a number $b\in \mathbb{N}$. We will call a Young diagram $\lambda$ $b$-regular if there exist no $i\in \mathbb{N}$ such that $\lambda_i=\lambda_{i+1}=...=\lambda_{i+b-1}>0$. Also for a box $A=(i,j)$ the \textit{residue} of $A$ with respect to $b$, denoted by $\res A$, is the residue class $(j-i)$ mod $b$. 
\begin{Definition}
\label{irregular}
Given a partition $\lambda$ and a positive integer $b$, $\lambda$ can be uniquely written as a union of multisets $$\lambda=\nu\cup\mu$$ where each part of $\nu$ has multiplicity less than $b$ and each part of $\mu$ has multiplicity being a multiple of $b$. Denote $\Reg_b(\lambda)=\nu$ as the regular part of $\lambda$ and the irregular part $\Irr_b(\lambda)$ is defined by $\mu=b\star\Irr_b(\lambda)$, where the operator $b\,\star$ is to repeat each part of the partition $b$ times. This decomposition is called the $b\con$regular decomposition of $\lambda$.
\end{Definition}

Next, we define the core and quotient of a partition following \cite{haiman2002combinatorics}. 

\begin{Definition}
\label{haimancore}

A partition $\lambda$ is a $b$-core if it does not contain any ribbon of length $b$. The $b\con$core $\Core_b(\lambda)$ of any partition $\lambda$ is the partition that remains after one removes as many $b\con$ribbons in succession as possible. The result is independent of choices of removals.

\end{Definition}

\begin{Definition}
\label{haimanquotient}
For any box $A=(i,j)\in\lambda$, let $B$ and $C$ be the boxes at the end of the arm $a_{ij}$ and the leg $l_{ij}$ respectively. Then $H_{ij}$ is divisible by $b$ precisely when res $B=k$ and res $C=k+1$ for some $k\in\{0,1,...,b-1\}$. Now for some fixed $k$ the boxes $A$ with res $B=k$ and  $\res C=k+1$ form an "exploded" copy of a partition which we denote $\lambda_k$. The quotient of a partition $\lambda$ is defined to be the $b\con$tuple of partitions, $\Quot_b(\lambda) =(\lambda_0,\lambda_1,...,\lambda_{b-1})$.
%Haimanquotient
%cores and quotients

\end{Definition}

\begin{Example}
Let $\lambda=(6,5,3,3,2,1,1)$ and $b=4$, and the residue of the rim of $\lambda$ are labelled in the picture.

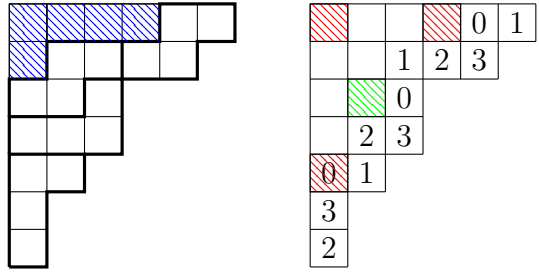
\begin{center}
\begin{tikzpicture}[scale=0.5]
\draw (0,0) -- (0,7);
\draw (1,0) -- (1,7);
\draw (2,2) -- (2,7);
\draw (3,3) -- (3,7);
\draw (5,5) -- (5,7);
\draw (6,6) -- (6,7);
\draw (0,0) -- (1,0);
\draw (0,1) -- (1,1);
\draw (0,2) -- (2,2);
\draw (0,3) -- (3,3);
\draw (0,4) -- (3,4);
\draw (0,5) -- (5,5);
\draw (0,6) -- (6,6);
\draw (0,7) -- (6,7);
\draw (4,5) -- (4,7);
\draw[line width=1.2pt] (0,0) -- (1,0) -- (1,2) --(2,2) -- (2,3) -- (0,3) -- (0,0);
\draw[line width=1.2pt] (0,3) -- (0,4) --(2,4) -- (2,5) -- (3,5) -- (3,3) --(2,3);
\draw[line width=1.2pt] (3,5) -- (5,5) -- (5,6) --(6,6) -- (6,7) --(4,7) --(4,6) -- (3,6) -- (3,5);
\draw[line width=1.2pt] (0,4) -- (0,5) -- (1,5) -- (1,6) --(3,6);
\draw[pattern=north west lines, pattern color=blue] (0,5) rectangle (1,7);
\draw[pattern=north west lines, pattern color=blue] (1,6) rectangle (4,7);
%right
\draw (8,0) -- (8,7);
\draw (9,0) -- (9,7);
\draw (10,2) -- (10,7);
\draw (11,3) -- (11,7);
\draw (13,5) -- (13,7);
\draw (14,6) -- (14,7);
\draw (8,0) -- (9,0);
\draw (8,1) -- (9,1);
\draw (8,2) -- (10,2);
\draw (8,3) -- (11,3);
\draw (8,4) -- (11,4);
\draw (8,5) -- (13,5);
\draw (8,6) -- (14,6);
\draw (8,7) -- (14,7);
\draw (12,5) -- (12,7);
\foreach \x/\y/\m in {+8.5/+0.5/$2$,8.5/1.5/$3$,8.5/2.5/$0$,9.5/2.5/$1$,9.5/3.5/$2$,10.5/3.5/$3$,10.5/4.5/$0$,10.5/5.5/$1$,11.5/5.5/$2$,12.5/5.5/$3$,12.5/6.5/$0$,13.5/6.5/$1$} 
    \node at (\x,\y) {\m};
\draw[pattern=north west lines, pattern color=red] (8,2) rectangle (9,3);
\draw[pattern=north west lines, pattern color=red] (8,6) rectangle (9,7);
\draw[pattern=north west lines, pattern color=red] (11,6) rectangle (12,7);
\draw[pattern=north west lines, pattern color=green] (9,4) rectangle (10,5);

\end{tikzpicture}
\end{center}

After removing the $4$ pieces of $4\con$ribbons, we obtain $\Core_4(\lambda)=(4,1)$. $\Quot_4(\lambda)=((1),(2,1),\emptyset,\emptyset)$, as shown in the above picture.

\end{Example}

\subsection{Two Equivalent Definitions of Mullineux Transpose}

We abbreviate the composition of Mullineux involution (see Definition \ref{RepMull}) and transpose as Mullineux transpose. Now we define the notion of good and co-good boxes as well as good and co-good sequence, which will be used to give the construction of Mullineux transpose with respect to some $b\in\mathbb{N}_{>1}$.

\begin{Definition}
\label{good}
A good box of residue $i$ where $i\in\{0,1,...,b-1\}$ of a partition $\lambda$ is defined through the following procedure:

First label the boxes on the boundary of $\lambda$ by their residues. Then moving from southwest to northeast, we produce a word by writing "R" for the removable boxes and "A" for addable boxes of residue $i$ (ignoring boxes in the boundary of other residues), thus obtaining a sequence, which is called an RA-sequence. Then we inductively cancel the consecutive "RA"'s until there is no "RA" appearing. Then the removable box of residue $i$ corresponding to first "R" from left is called a good box of residue $i$. If there are no "R" in the word after the cancellation, there is no good box of residue $i$.
\end{Definition}

\begin{Remark}
This definition is equivalent to Kleshchev's original definition in \cite{kleshchev1996branching}. It follows that for each residue $i=0,1,...,b-1$, there is at most one good box and Kleshchev proved there is always a good box of some residue.
\end{Remark}

\begin{Definition}
\label{cogood}
A co-good box of residue $i$ for $i\in\{0,1,...,b-1\}$ of a partition $\lambda$ is defined through the following procedure:

Label the boxes of the boundary of $\lambda$ by their residue and for a given residue $i$, write its corresponding RA sequence. Then we cancel the consecutive "AR"'s iteratively until there is no "AR" appearing. Now the removable box of residue $i$ corresponding to the first "R" from right is called an co-good box of residue $i$. (As with good boxes, for each $i\in \{0,1,...,b-1\}$ there exists at most one co-good box and there is at least one value of $i$ for which such a box exists.) 
\end{Definition}

\begin{Remark}
It is clear from the definitions that if $A=(i,j)$ is a good box for $\lambda$ with $\res A=k$ mod $b$, then $A'=(j,i)$ is a co-good box for $\lambda^{\Tr}$ with $\res A'=-k$, and vice versa. 
\end{Remark}

\begin{Definition}
For $\lambda$ a $b$-regular partition on $n$, a sequence $(r_1,...,r_n)\pmod{b}$ of residues is called \textit{good} (resp. \textit{co-good}) if\\
$\lambda$ has a good (resp. co-good) box $A_1$ of residue $r_1$,\\
$\lambda\setminus A_1$ has a good (resp. co-good) box $A_2$ of residue $r_2$,\\...\\
$\lambda\setminus\bigcup\limits_{i=1}^{n-1} A_i$ has a good (resp. co-good) box $A_n$ of residue $r_n$.

We call the sequence $A_1,...,A_n$ as the good (resp. co-good) decomposition sequence of $\lambda$.
\end{Definition}

Then we consider the Mullineux involution $\Mb$. Throughout the paper we will deal with the composition of Mullineux map with transpose $\MbT$ rather than $\Mb$ alone.

The following Theorem is a reformulation of \cite[Theorem 6.42]{mathas1999iwahori} due to Klechshev and Brundan, where it gives an combinatorial way to do Mullineux transpose, and $b$ is not necessarily restricted to be prime. More importantly, it gives an equivalent definition of Mullineux involution to Definition \ref{RepMull}.

\begin{Theorem}
\label{goodco-goodmt}

For any $b\con$regular partition $\lambda$, consider the following procedure:

\begin{enumerate}
\item Find a good box $A_1$ for $\lambda$, record its residue $\res A_1=r_1$ and delete the box from the partition to obtain a smaller partition $\lambda_1$. Repeat the above step $n$ times until we get the empty partition $\lambda_n=\emptyset$. Then by construction the sequence $(r_1,...,r_n)$ is a good sequence of $\lambda$.
\item Start with the empty partition 
$\mu_n=\emptyset$ and at step $i$ add the unique box $B_{n-i+1}$ to the partition $\mu_{n-i+1}$ such that $B_{n-i+1}$ is a co-good square of $\mu_{n-i+1}\cup B_{n-i+1}$ of residue $r_{n-i+1}$. Such a box can always be found uniquely. Label the resulting partition by $\mu=\mu_1\cup B_1$. 
\end{enumerate}
Then $\mu=\lambda^{\MbT}$.
\end{Theorem}
\begin{Example}
Consider $\lambda=(5,4,2)$ and $b=4 $.
We label the boxes in $\lambda$ and their residues as follows:
\begin{center}
\begin{tikzpicture}[scale=0.6]
\draw (2,0) -- (0,0) -- (0,3);
\draw (1,0) -- (1,3);
\draw (2,0) -- (2,3);
\draw (3,1) -- (3,3);
\draw (0,1) -- (4,1) -- (4,3);
\draw (0,2) -- (5,2) -- (5,3) -- (0,3);

\draw (10,0) -- (8,0) -- (8,3);
\draw (9,0) -- (9,3);
\draw (10,0) -- (10,3);
\draw (11,1) -- (11,3);
\draw (8,1) -- (12,1) -- (12,3);
\draw (8,2) -- (13,2) -- (13,3) -- (8,3);

\foreach \x/\y/\m in {0.5/2.5/$A_{11}$,1.5/2.5/$A_9$,2.5/2.5/$A_8$,3.5/2.5/$A_6$,4.5/2.5/$A_1$,0.5/1.5/$A_{10}$,1.5/1.5/$A_7$,2.5/1.5/$A_5$,3.5/1.5/$A_4$,0.5/0.5/$A_{3}$,1.5/0.5/$A_{2}$,8.5/2.5/$0$,9.5/2.5/$1$,10.5/2.5/$2$,11.5/2.5/$3$,12.5/2.5/$0$,8.5/1.5/$3$,9.5/1.5/$0$,10.5/1.5/$1$,11.5/1.5/$2$,8.5/0.5/$2$,9.5/0.5/$3$} 
    \node at (\x,\y) {\m};
\end{tikzpicture}
\end{center}

By Definition \ref{good}, we decompose $\lambda$ as $$A_1,A_{2},A_{3},A_4,A_5,A_6,A_7,A_8,A_9,A_{10},A_{11},$$ with good sequence $(0,3,2,2,1,3,0,2,1,3,0)$.

\begin{center}
\begin{tikzpicture}[scale=0.6]
\draw (0,0) -- (1,0) -- (1,5);
\draw (0,1) -- (2,1) -- (2,5);
\draw (0,2) -- (2,2);
\draw (0,3) -- (2,3);
\draw (0,4) -- (4,4) -- (4,5) -- (0,5);
\draw (0,0) -- (0,5);
\draw (3,4) -- (3,5);

\draw (8,0) -- (9,0) -- (9,5);
\draw (8,1) -- (10,1) -- (10,5);
\draw (8,2) -- (10,2);
\draw (8,3) -- (10,3);
\draw (8,4) -- (12,4) -- (12,5) -- (8,5);
\draw (8,0) -- (8,5);
\draw (11,4) -- (11,5);

\foreach \x/\y/\m in {0.5/4.5/$B_{11}$,1.5/4.5/$B_9$,2.5/4.5/$B_3$,3.5/4.5/$B_2$,0.5/3.5/$B_{10}$,1.5/3.5/$B_7$,0.5/2.5/$B_8$,1.5/2.5/$B_6$,0.5/1.5/$B_5$,1.5/1.5/$B_{4}$,0.5/0.5/$B_{1}$,8.5/4.5/$0$,9.5/4.5/$1$,10.5/4.5/$2$,11.5/4.5/$3$,8.5/3.5/$3$,9.5/3.5/$0$,8.5/2.5/$2$,9.5/2.5/$3$,8.5/1.5/$1$,9.5/1.5/$2$,8.5/0.5/$0$} 
    \node at (\x,\y) {\m};
\end{tikzpicture}
\end{center}

Using Definition \ref{cogood} and the same sequence $(0,3,2,2,1,3,0,2,1,3,0)$ as a co-good sequence, partition $(4,2,2,2,1)$ can be rebuild using the co-good decomposition sequence $$B_{1},B_2,B_3,B_{4},B_5,B_6,B_7,B_8,B_9,B_{10},B_{11}.$$

Hence by Theorem \ref{goodco-goodmt}, $(5,4,2)^{\MfourT}=(4,2,2,2,1)$.

\end{Example}

Bessenrodt, Olsson and Xu introduced in \cite{bessenrodt1999properties} another equivalent definition of Mullineux transpose in the following way, which is used in Section \ref{secunique} of the monotonicity properties.

First we define the $b\con$rim of a $b\con$regular partition and the operator $\I$ of removing the $b\con$rim.
\begin{Definition}
For a $b\con$regular partition $\lambda$, the $b\con$rim of $\lambda$ is defined to be a subset of the rim consisting of the following pieces. Each piece, \textit{except possibly the last one}, contains $b$ boxes. We choose the first $b$ boxes from the rim, beginning with the rightmost box of the first row and moving southwestwards. If the last box of this piece is chosen from the $i_0\con$th row of $\lambda$, then we choose the second piece of $b$ boxes beginning with the rightmost box of the next row ${i_0+1}$. Continue this procedure until we reach the last piece ending in the last row. Define $\lambda^{\I}$ to be the partition obtained from $\lambda$ by removing its $b\con$rim.
\end{Definition}
Next we define an operator $\J$ for a $b$-regular partition $\lambda$. 
\begin{Definition}
\label{defJ}
Given $\lambda=(\lambda_1,...,\lambda_k)$, if $\lambda^{\I}=(\mu_1,...,\mu_k)$, where some of the $\mu_i$ in the end are allowed to be zero, and $\phi(\lambda)=|\lambda|-|\lambda^{\I}|$, define $$\lambda^{\J}\coloneqq(\mu_1+1,...,\mu_{k-1}+1,\mu_k+\delta)$$ where 
\[\delta=
\begin{cases}
0\text{ if } b\nmid \phi(\lambda)\\
1\text{ if } b\mid \phi(\lambda)
\end{cases}\]
\end{Definition}

Finally, the operator $\X_b$ for a $b\con$regular partition $\lambda$ is defined as $\lambda^{\X_b}\coloneqq(j_1,...,j_l)$, where 
$$j_i=|\lambda^{\J^{i-1}}|-|\lambda^{\J^i}|.$$

\begin{Proposition}[{\cite[Proposition 3.6]{bessenrodt1999properties}}]
For any $b\con$regular partition $\lambda$, we have
$$\lambda^{\Xb}=\lambda^{\MbT}.$$
\end{Proposition}

We extend the definition of Mullineux transpose to all partitions using the $b\con$regular decomposition in Definition \ref{irregular} as follows.

\begin{Definition}
\label{extMT}
The extended Mullineux transpose transformation $\mathcal{W}_b:\mathcal{P}\rightarrow\mathcal{P}$ is defined to be $$\lambda^{\mathcal{W}_b}\coloneqq (\nu^{\Mb}\cup b\star \mu^{\Tr})^{\Tr}$$ where $\lambda=\nu\cup b\star \mu$ is the $b\con$regular decomposition. In particular, if $\lambda$ is $b\con$regular, then $\lambda^{\mathcal{W}_b}=\lambda^{\MbT}$.
\end{Definition}

The following lemma is some basic properties of a core.

\begin{Lemma}
\label{RAlemma}
Let $\lambda$ be a $b\con$core, then
\begin{enumerate}
\item $\lambda^{\Tr}$ is also a $b\con$core;

\item The $b\con$rim and the rim of $\lambda$ coincide;

\item Given any residue in $\{0,1,...,b-1\}$, the RA-sequence of this residue contains only A's or R's.
\end{enumerate}
\end{Lemma}
\begin{proof}
(1) and (2) are straightforward from the definition of a core.

(3)
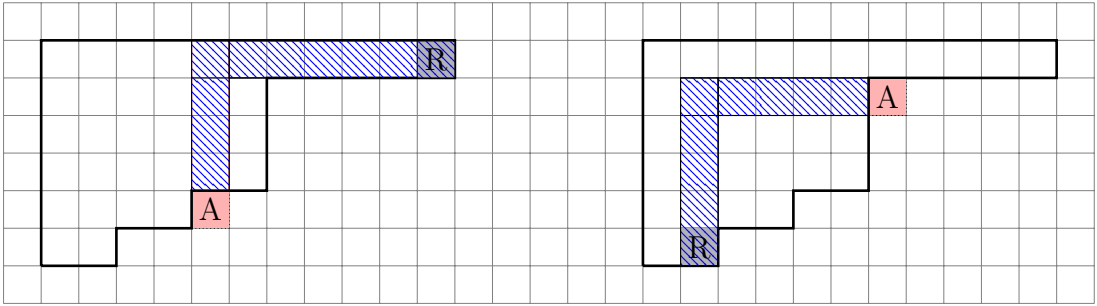
\begin{center}
\begin{tikzpicture}
\draw[step=0.5cm,gray,very thin] (0.5,0) grid (15,4);
\draw[line width=1pt] (1,0.5) -- (2,0.5) -- (2,1) --(3,1) -- (3,1.5) --(4,1.5) -- (4,3) --(6.5,3) --(6.5,3.5) -- (1,3.5) -- (1,0.5);
\draw[line width=1pt] (9,0.5) -- (10,0.5) -- (10,1) --(11,1) -- (11,1.5) --(12,1.5) -- (12,3) --(14.5,3) --(14.5,3.5) -- (9,3.5) -- (9,0.5);
\draw[densely dotted] (3,1) -- (3.5,1) -- (3.5,1.5);
\draw (6,3) -- (6,3.5);
\fill[red!30!white] (3.02,1.01) rectangle (3.49,1.48);
\fill[black!30!white] (6.01,3.01) rectangle (6.49,3.48);
\draw[red] (3,1.5) -- (3,3.5);
\draw[red] (3.5,1.5) -- (3.5,3.5);
\draw[black] (3,3) -- (4,3);
\draw[pattern=north west lines, pattern color=blue] (3,1.5) rectangle (3.5,3.5);
\draw[pattern=north west lines, pattern color=blue] (3.5,3) rectangle (6.5,3.5);
\foreach \x/\y/\m in {+3.25/+1.25/A,6.25/3.25/R} 
    \node at (\x,\y) {\m};
%right one
\draw[densely dotted] (12,2.5) -- (12.5,2.5) -- (12.5,3);
\draw (9.5,0.5) -- (9.5,1) -- (10.5,1);
\fill[red!30!white] (12.02,2.51) rectangle (12.49,2.98);
\fill[black!30!white] (9.5,0.52) rectangle (9.98,1);
%\draw[red] (3,1.5) -- (3,3.5);
%\draw[red] (3.5,1.5) -- (3.5,3.5);
\draw[black] (9.5,0.5) -- (9.5,3) -- (12,3);
\draw[black] (12,2.5) -- (9.5,2.5);
\draw[black] (10,3) -- (10,1);
\draw[pattern=north west lines, pattern color=blue] (9.5,0.5) rectangle (10,3);
\draw[pattern=north west lines, pattern color=blue] (10,2.5) rectangle (12,3);
\foreach \x/\y/\m in {+9.75/+0.75/R,12.25/2.75/A} 
    \node at (\x,\y) {\m};
\end{tikzpicture}
\end{center}

If the RA sequence contains $\cdots \rm{A} \cdots\cdots \rm{R}\cdots$ or $\cdots \rm{R} \cdots\cdots \rm{A}\cdots$, then the specific hook corresponding to these addable and removable boxes has length divisible by $b$, which contradicts $\lambda$ be a $b\con$core.
\end{proof}

\subsection{Column Regularization}
For an arbitrary partition $\lambda$, \cite{james16representation} defined $b\con$regularization of $\lambda$ as sliding the boxes of $\lambda$ upwards on all ladders of  slope $-\frac{1}{p-1}$, and denote the resulting partition by $\lambda^{\reg_b}$, which is $b\con$regular. And column regularization $\lambda^{\Cr_b}$ is defined to be $((\lambda^{\Tr})^{\reg_b})^{\Tr}$. This is a special case of our generalized column regularization with two parameters.

\begin{Definition}
\label{colreg}
For two co-prime nonnegative integers $0<a<b$, we define a partial transformation $\Crab:\mathcal{P}\rightarrow\mathcal{P}$ as follows.

For any partition $\lambda$ (identifying as a set of integer points in the plane), ladders are defined as lines $$L_c: \; y+\frac{b-a}{a}x=\frac{c}{a}$$ 
with $c\in\mathbb{Z}$, and identify $L_c$ with the set of integer points on it. Denote $L_c^+=L_c\cap\{(x,y)|x>0,y>0\}$. For each $c\in\mathbb{Z}$, if $\lambda\cap L_c^+\neq\emptyset$, slide those boxes in the intersection down the ladder $L_c^+$ to the bottom. The resulting set of boxes is $\lambda^{\Crab}$, which may form a partition or not.
\end{Definition}
%\begin{Remark}
%Note $\Crab$ is only a partial transformation, since after the procedure done to $\lambda$ in Definition \ref{colreg}, the collection of boxes may not form a partition. 
%\end{Remark}
\begin{Remark}
To avoid confusion with the usual picture in mind, we restate the fact which is already mentioned in Section \ref{secpre} that our $x\con$axis is pointing southwards and $y\con$axis is pointing eastwards. This is because we need to be consistent with the notion of the coordinate of a box in the partition, where the first coordinate is the corresponding row index and the second being the column index. And this convention is used throughout the paper.

In fact, sliding a box on a ladder with parameter $a$, $b$ is to slide it $ta$ spaces down and $t(b-a)$ spaces to the left, where $t\in\mathbb{N}_{>0}$.
\end{Remark}
\begin{Remark}
\label{colregb}
The previous definition satisfies $\lambda^{\Cr_b}=\lambda^{\Cr_{1,b}}$.
\end{Remark}
\begin{Example}
$(3,2,2,1)^{\Cr_{2,3}}=(2,2,2,1,1)$ where box $C$ slides to where $B$ is and box $B$ slides down to the position $A$, as shown in the picture. 

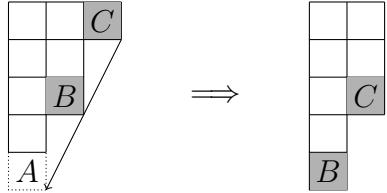
\begin{center}
\begin{tikzpicture}[scale=0.5]
\draw (1,0) -- (0,0) -- (0,4);
\draw (1,0) -- (1,4);
\draw (2,1) -- (2,4);
\draw (3,3) -- (3,4);
\draw (0,1) -- (2,1);
\draw (0,2) -- (2,2);
\draw (0,3) -- (3,3);
\draw (0,4) -- (3,4);
\draw[densely dotted] (0,0) -- (0,-1) -- (1,-1) -- (1,0);
\draw[->] (3,3) -- (1,-1);
\draw (9,0) -- (8,0) -- (8,4);
\draw (9,0) -- (9,4);
\draw (10,1) -- (10,4);
\draw (8,1) -- (10,1);
\draw (8,2) -- (10,2);
\draw (8,3) -- (10,3);
\draw (8,4) -- (10,4);
\draw (8,0) -- (8,-1) -- (9,-1) -- (9,0);
\fill[black!30!white] (1.02,1.02) rectangle (1.98,1.98);
\fill[black!30!white] (2.02,3.02) rectangle (2.98,3.98);
\fill[black!30!white] (9.02,1.02) rectangle (9.98,1.98);
\fill[black!30!white] (8.02,-0.98) rectangle (8.98,-0.02);
\foreach \x/\y/\m in {0.5/-0.5/$A$,1.5/1.5/$B$,2.5/3.5/$C$,5.5/1.5/$\Longrightarrow$,8.5/-0.5/$B$,9.5/1.5/$C$} 
    \node at (\x,\y) {\m};
\end{tikzpicture}
\end{center}

However, after applying $\Cr_{2,3}$ to $(3,2,2)$, it is not a partition any more. The box $C$ slides to where $B$ is,  $B$ slides to position $A$, and $F$ slides to position $E$, shown in the picture below.

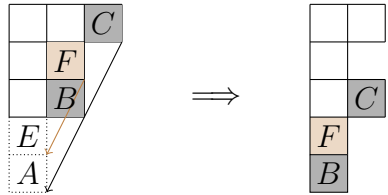
\begin{center}
\begin{tikzpicture}[scale=0.5]
\draw (0,1) -- (0,4);
\draw (1,1) -- (1,4);
\draw (2,1) -- (2,4);
\draw (3,3) -- (3,4);
\draw (0,1) -- (2,1);
\draw (0,2) -- (2,2);
\draw (0,3) -- (3,3);
\draw (0,4) -- (3,4);
\draw[densely dotted] (0,1) -- (0,-1) -- (1,-1) -- (1,1);
\draw[densely dotted] (0,0) -- (1,0);
\fill[black!30!white] (1.02,1.02) rectangle (1.98,1.98);
\fill[black!30!white] (2.02,3.02) rectangle (2.98,3.98);
\fill[black!30!white] (9.02,1.02) rectangle (9.98,1.98);
\fill[black!30!white] (8.02,-0.98) rectangle (8.98,-0.02);
\fill[brown!30!white] (1.02,2.02) rectangle (1.98,2.98);\fill[brown!30!white] (8.02,0.02) rectangle (8.98,0.98);\draw[->] (3,3) -- (1,-1);
\draw[->,brown] (2,2) -- (1,0);
\foreach \x/\y/\m in {0.5/-0.5/$A$,1.5/1.5/$B$,2.5/3.5/$C$,1.5/2.5/$F$,0.5/0.5/$E$,5.5/1.5/$\Longrightarrow$,8.5/-0.5/$B$,9.5/1.5/$C$,8.5/0.5/$F$} 
    \node at (\x,\y) {\m};
\draw (9,0) -- (8,0) -- (8,4);
\draw (9,0) -- (9,4);
\draw (10,1) -- (10,2);
\draw (10,3) -- (10,4);
\draw (8,1) -- (10,1);
\draw (8,2) -- (10,2);
\draw (8,3) -- (10,3);
\draw (8,4) -- (10,4);
\draw (8,0) -- (8,-1) -- (9,-1) -- (9,0);
\end{tikzpicture}
\end{center}

\end{Example}

\section{Two Series of Transformations and the Main Theorem}
\label{mainsection}
We first define the combinatorial wall-crossing transformation due to Bezrukavnikov and Losev. For $\lambda$ being a partition of $n$, consider the Farey sequence $F_n$ which is a set of reduced fractions between $0$ and $1$ with denominator at most $n$. Each $\frac{a}{b}\in F_n$ is called a wall. Every two consecutive elements $\frac{a_i}{b_i}$ and $\frac{a_{i+1}}{b_{i+1}}$ (reduced fractions) of this sequence define an interval $I=\left[\frac{a_i}{b_i}, \frac{a_{i+1}}{b_{i+1}}\right]\subset [0,1]$. We define the wall-crossing transformation at the wall $\frac{a}{b}\in F_n$, where $\frac{a}{b}$ is reduced, to be the extended Mullineux transpose $\mathcal{W}_b:\mathcal{P}\rightarrow\mathcal{P}$ given in Definition \ref{extMT}. Now we compose them in the following way:

\begin{Definition}
\label{CWC}
Fix a positive integer $n$. We define a collection of maps $B_I:\mathcal{P}_n\rightarrow\mathcal{P}_n$ where $I$ are all intervals whose endpoints are consecutive rational numbers in $F_n$.

For the first interval, $B_{\left[0,\frac{1}{n}\right]}(\lambda)=\lambda$ for every $\lambda\in \mathcal{P}$. Inductively, suppose we already defined $B_I$ where $I=\left[\frac{a_{i-1}}{b_{i-1}}, \frac{a_i}{b_i}\right]$. Suppose the next interval is $I^\prime=\left[\frac{a_{i}}{b_{i}}, \frac{a_{i+1}}{b_{i+1}}\right]$, we define $B_{I^\prime}(\lambda)\coloneqq B_I(\lambda)^{\mathcal{W}_{b_i}}$. Also we define an integer function $D_I(\lambda)\coloneqq b_i\cdot|\Irr_{b_i}(B_{I}(\lambda))|$ when $I=\left[\frac{a_{i-1}}{b_{i-1}}, \frac{a_i}{b_i}\right]$.
\end{Definition}

\begin{Remark}
\label{Bbij}
For each interval $I$, $B_I:\mathcal{P}_n\rightarrow\mathcal{P}_n$ is a bijection.
\end{Remark}

In fact, we can consider the process of starting with any $\lambda$ in $\left[0,\frac{1}{n}\right]$, and do a series of wall-crossing transformations $\mathcal{W}_b$, then $\{B_I(\lambda)\}_I$ gives a series of partitions, one in each interval and one is obtained from the previous one by crossing a wall via combinatorial wall-crossing.

Moreover, consider another procedure where we begin with $\lambda$ in the first interval, and cross the wall $\frac{a}{b}\in F_n$ by performing the generalized column regularization $\Cr_{a,b}$ to the partition in the previous interval. We denote the corresponding partition in $I$ by $\widetilde{B_{I}}(\lambda)$. Since $\Cr_{a_i,b_i}$ is only a partial transformation, at the moment we cannot guarantee the validity of doing such a process throughout the unit interval. But fortunately we have the following lemma which guarantees $\widetilde{B_{I}}(\lambda)$ is well defined.

\begin{Lemma}
\label{valid}
If we start with any partition $\lambda$ of $n$, at each step of the second procedure (performing generalized column regularization), we have $\widetilde{B_{I}}(\lambda)\in\mathcal{P}_n$. In particular, $\widetilde{B_{\left[\frac{n-1}{n},1\right]}}(\lambda)=(1^n)$.
\end{Lemma}
\begin{proof}
Suppose at some step of the process, $\widetilde{B_{I}}(\lambda)\notin\mathcal{P}_n$. Denote $I^\prime=\left[\frac{a^\prime}{b^\prime},\frac{a}{b}\right]$ to be the previous interval. First of all $I^\prime$ cannot be the first interval $\left[0,\frac{1}{n}\right]$ since by Remark \ref{colregb} and \cite{james16representation}, $\lambda^{\Cr_b}=\lambda^{\Cr_{1,b}}$ is always well-defined.

Then there is a box $\overline{A}=(i,j)\in\widetilde{B_{I}}(\lambda)$, but the box on top of it $B=(i-1,j)\notin\widetilde{B_{I}}(\lambda)$. This situation happens only when $\overline{A}$ comes from some $A\in \widetilde{B_{I^\prime}}(\lambda) $ with $\Slope(A,\overline{A})=-\frac{b-a}{a}$. Also $B\in \widetilde{B_{I^\prime}}(\lambda) $ and slides down to some $\overline{B}\in \widetilde{B_{I}}(\lambda) $ with $\Slope(B,\overline{B})=-\frac{b-a}{a}$.

Next we know 
$$\max_{\square\in\widetilde{B_{I^\prime}}(\lambda)}\frac{a_\square}{l_\square+1}<\frac{b^\prime-a^\prime}{a^\prime}.$$
This is immediate from the definition of column regularization, where we perform $\Cr_{c_k,d_k}$ in order to the initial partition $\lambda$, with $\frac{c_1}{d_1}=\frac{1}{n}<...<\frac{c_m}{d_m}=\frac{a^\prime}{b^\prime}$. After each $\Cr_{c_k,d_k}$, the ratio of arm length and leg length plus one must be strictly smaller than $\frac{d_k-c_k}{c_k}$ because all the possible shallower slopes are removed in previous steps. Hence the inequality is true since $m\geq1$.

Therefore $A$ and $B$ are removable boxes and $\overline{A}$ and $\overline{B}$ are addable boxes in $\widetilde{B_{I^\prime}}(\lambda)$. If not, then we are able to find some removable box $C$ southeast to $A$ (resp. $B$) and some addable $\overline{C}$ northwest to $\overline{A}$ (resp. $\overline{B}$) such that $-\Slope(C,\overline{C})>\frac{b-a}{a}$, i.e. $-\Slope(C,\overline{C})\geq\frac{b^\prime-a^\prime}{a^\prime}$, which is a contradiction. Say $A$ and $\overline{A}$ correspond to a box with arm length $t(b-a)$ and leg length $ta-1$; $B$ and $\overline{B}$ correspond to a box with arm length $t^\prime(b-a)$ and leg length $t^\prime a-1$.

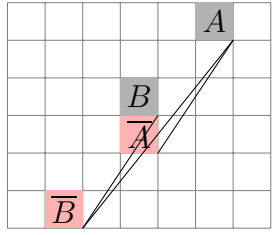
\begin{center}
\begin{tikzpicture}
\draw[step=0.5cm,gray,very thin] (0,0.5) grid (3.5,3.5);
\fill[red!30!white] (0.5,0.5) rectangle (1,1);
\fill[red!30!white] (1.5,1.5) rectangle (2,2);
\fill[black!30!white] (1.5,2) rectangle (2,2.5);
\fill[black!30!white] (2.5,3) rectangle (3,3.5);
\foreach \x/\y/\m in {+0.75/+0.75/$\overline{B}$,+1.75/2.25/$B$,+1.75/+1.75/$\overline{A}$,+2.75/+3.25/$A$} 
    \node at (\x,\y) {\m};
    \draw (1,0.5) -- (2,2);
    \draw (2,1.5) -- (3,3);
    \draw (1,0.5) -- (3,3);
\end{tikzpicture}
\end{center}

But now we have $$-\Slope(A,\overline{B})=\frac{t(b-a)+t^\prime(b-a)}{ta+t^\prime a-1}>\frac{b-a}{a}.$$
Hence $A$ and $\overline{B}$ correspond to a box in $\widetilde{B_{I^\prime}}(\lambda)$ with the ratio of arm and leg plus one larger or equal to $\frac{b^\prime-a^\prime}{a^\prime}$, which contradicts the above inequality.

Using similar argument as above, when we arrive at the last interval $\left[\frac{n-1}{n},1\right]$, all possible slopes are removed except the steepest slope $0$, hence $\widetilde{B_{\left[\frac{n-1}{n},1\right]}}(\lambda)=(1^n)$.
\end{proof}

In Step 1 of the proof of Theorem \ref{main}, we provide a simpler proof to Lemma \ref{valid} in case of $\lambda=(n)$. What we are going to present is that when $\lambda$ is the one-row partition $(n)$, the above two procedures have the same effect. For simplicity, from now on, we will denote 
$$\lambda_I\coloneqq B_I((n)),$$
 $$\widetilde{\lambda_I}\coloneqq \widetilde{B_I}((n)).$$

\begin{Example}
In case of $n=7$, the Farey sequence is 
$$\frac{1}{7},\frac{1}{6},\frac{1}{5},\frac{1}{4},\frac{2}{7},\frac{1}{3},\frac{2}{5},\frac{3}{7},\frac{1}{2},\frac{4}{7},\frac{3}{5},\frac{2}{3},\frac{5}{7},\frac{3}{4},\frac{4}{5},\frac{5}{6},\frac{6}{7} ,$$and we start with $(7)$ in $\left[0, \frac{1}{7}\right]$.

Then either of the above two procedures give the same sequence of partitions in each interval, as shown in the following table (for simplicity, if the partitions in consecutive intervals are the same, we just write them once by pointing out the union of those small intervals):

\renewcommand\arraystretch{1.5}
\begin{table}[!htb]
\centering
\caption {n=7, starting with $(n)$}
\begin{tabular}{|c|c|c|c|c|c|c|c|c|}
\hline
Intervals & $\left[0,\frac{1}{7}\right]$ & $\left[\frac{1}{7},\frac{1}{5}\right]$ & $\left[\frac{1}{5},\frac{1}{3}\right]$ & $\left[\frac{1}{3},\frac{1}{2}\right]$ & $\left[\frac{1}{2},\frac{2}{3}\right]$ & $\left[\frac{2}{3},\frac{4}{5}\right]$ & $\left[\frac{4}{5},\frac{6}{7}\right]$ & $\left[\frac{6}{7},1\right]$ \\
\hline
Partitions & $(7)$ & $(6,1)$ & $(5,2)$ & $(4,2,1)$ & $(3,2,1^2)$ & $(2^2,1^3)$ & $(2,1^5)$ & $(1^7)$ \\
\hline
\end{tabular}
\end{table} 
\end{Example}

Now for the combinatorial wall-crossing operation, we denote the endpoints where the partition is not identical in the consecutive two intervals (sharing the endpoints) by $p_0=0<p_1<...<p_s$, and call them breaks for the combinatorial wall-crossing. And denote $\lambda_k$ to be the partition in $\left[p_{k-1},p_k\right]$. Similarly, we denote $q_0=0<q_1<...<,q_t$ to be the breaks of column regularization. And let $\widetilde{\lambda_k}$ be the partition in $[q_{k-1},q_k]$ under a series of column regularization.

\begin{Theorem}[Main result of the paper]
\label{main}
Using the above notation, we have 
\begin{enumerate}
\item $s=t$ and $$p_k=q_k=\min_{(i,j)\in\lambda_k}\frac{l_{ij}+1}{H_{ij}}=\max_{(i,j)\in\lambda_{k+1}}\frac{l_{ij}}{H_{ij}}.$$

\item $\lambda_k=\widetilde{\lambda_k}$ for all $k$, i.e. the two operations are exactly the same when we start with the row partition $(n)$.
\end{enumerate}
\end{Theorem}

Before proving Theorem \ref{main}, we state the following result from \cite{bessenrodt1999properties}.

\begin{Proposition}[\cite{bessenrodt1999properties}]
\label{MTcore}
For a $b\con$regular partition, $\lambda^{\MbT}=\lambda$ if and only if $\lambda$ is a $b\con$core.
\end{Proposition}

\begin{proof}[Proof of Theorem \ref{main}]
First of all, by direct computation, $$(n)^{\Cr_{1,n}}=(n)^{\MnT}=(n-1,1);$$ $$\frac{1}{n}=\min_{(i,j)\in(n)}\frac{l_{ij}+1}{H_{ij}}=\max_{(i,j)\in(n-1,1)}\frac{l_{ij}}{H_{ij}}$$ when $n\geq2$. Then we induct on $k$ and suppose that until $p_k=q_k$, the two operations are exactly the same and the breaks satisfy the property in the theorem. In particular, $\lambda_{k+1}=\widetilde{\lambda_{k+1}}$ and $$p_k=q_k=\max_{(i,j)\in\lambda_{k+1}}\frac{l_{ij}}{H_{ij}}=\frac{a_k}{b_k}.$$ In the meantime, we induct on the fact that $\lambda_k$ is $b_k\con$regular.

Then we need to prove the following two claims:
\begin{itemize}
\item For any reduced fraction $\frac{a}{b}$ satisfying
\begin{gather}
\max_{(i,j)\in\lambda_{k+1}}\frac{l_{ij}}{H_{ij}}<\frac{a}{b}<\min_{(i,j)\in\lambda_{k+1}}\frac{l_{ij}+1}{H_{ij}}, \tag{$*$} 
\end{gather}
 %$$\max_{(i,j)\in\lambda_{k+1}}\frac{l_{ij}}{H_{ij}}<\frac{a}{b}<\min_{(i,j)\in\lambda_{k+1}}\frac{l_{ij}+1}{H_{ij}}, $$
 there is $$\lambda_{k+1}^{\MbT}=\lambda_{k+1}^{\Crab}=\lambda_{k+1}.$$

\item Let $$\frac{a_{k+1}}{b_{k+1}}=\min_{(i,j)\in\lambda_{k+1}}\frac{l_{ij}+1}{H_{ij}}$$ be the reduced fraction. If $\frac{a_{k+1}}{b_{k+1}}=1$, then the process is ended, otherwise $\lambda_{k+1}$ is $b_{k+1}\con$regular and
 $$\lambda_{k+1}^{\mathcal{W}_{b_{k+1}}}=\lambda_{k+1}^{\MkT}=\lambda_{k+1}^{\Crabk}\neq\lambda_{k+1}.$$
\end{itemize}

For the part of $\lambda_{k+1}^{\MbT}=\lambda_{k+1}$ in the first claim, with $a$, $b$ satisfying $(*)$, we know that it suffices to prove $\lambda_{k+1}$ is a $b\con$core by Proposition \ref{MTcore}. 

Now suppose $\exists (i_0,j_0)\in \lambda_{k+1}$, such that $H_{i_0,j_0}=k_0 b$, where $k_0\in\mathbb{Z}_{>0}$, then
$$\frac{l_{i_0,j_0}}{k_0 b}=\frac{l_{i_0,j_0}}{H_{i_0,j_0}}\leq\max_{(i,j)\in\lambda_{k+1}}\frac{l_{ij}}{H_{ij}}<\frac{a}{b}<\min_{(i,j)\in\lambda_{k+1}}\frac{l_{ij}+1}{H_{ij}}\leq\frac{l_{i_0,j_0}+1}{H_{i_0,j_0}}=\frac{l_{i_0,j_0}+1}{k_0b}$$

hence, $$ k_0 a-1<l_{i_0,j_0}<k_0 a$$
and this leads to a contradiction since $l_{i_0,j_0}\in\mathbb{Z}$. Therefore $\lambda_{k+1}$ is a $b\con$core, when $$\max_{(i,j)\in\lambda_{k+1}}\frac{l_{ij}}{H_{ij}}<\frac{a}{b}<\min_{(i,j)\in\lambda_{k+1}}\frac{l_{ij}+1}{H_{ij}}$$.

In order to show that $\lambda_{k+1}^{\Crab}=\lambda_{k+1}$, with $a$, $b$ satisfying $(*)$, suppose there exists a ladder $$L_c: \; y+\frac{b-a}{a}x=\frac{c}{a}$$ such that $$L_c^+\cap\lambda_{k+1}\neq \emptyset$$ and the boxes in the intersection does not lie at the bottom of $L_c^+$, we do the following operations. Find $A\in L_c^+\cap\lambda_{k+1}$, $\overline{B}\in L_c^+\setminus\lambda_{k+1}$ and ${A}$ is above $\overline{B}$. Then pick ${A^\prime}$ which is a removable box and southeast to ${A}$ and an addable box $\overline{B^\prime}$ that is northwest to $\overline{B}$, as shown in the picture.

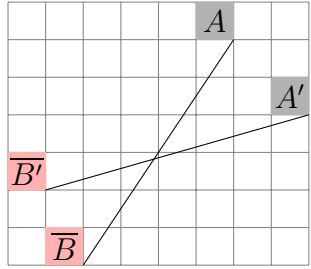
\begin{center}
\begin{tikzpicture}
\draw[step=0.5cm,gray,very thin] (0,0.5) grid (4,4);
\fill[red!30!white] (0.5,0.5) rectangle (1,1);
\fill[red!30!white] (0,1.5) rectangle (0.5,2);
\fill[black!30!white] (3.5,2.5) rectangle (4,3);
\fill[black!30!white] (2.5,3.5) rectangle (3,4);
\foreach \x/\y/\m in {+0.75/+0.75/$\overline{B}$,0.25/1.75/$\overline{B^\prime}$,+3.75/+2.75/$A^\prime$,+2.75/+3.75/$A$} 
    \node at (\x,\y) {\m};
    \draw (1,0.5) -- (3,3.5);
    \draw (0.5,1.5) -- (4,2.5);
\end{tikzpicture}
\end{center}

The following inequality holds:
$$\frac{b-a}{a}=-\Slope (A,\overline{B})\leq-\Slope (A^\prime,\overline{B^\prime})\leq \max_{(i,j)\in\lambda_{k+1}}\frac{a_{ij}}{l_{ij}+1}$$
i.e. $$\frac{a}{b}\geq\frac{1}{1+\max_{(i,j)\in\lambda_{k+1}}\frac{a_{ij}}{l_{ij}+1}}=\min_{(i,j)\in\lambda_{k+1}}\frac{l_{ij}+1}{H_{ij}}$$
which is a contradiction to $(*)$. Hence we obtain $\lambda_{k+1}^{\Crab}=\lambda_{k+1}$.

The second claim is equivalent to saying $\frac{a_{k+1}}{b_{k+1}}=\min_{(i,j)\in\lambda_{k+1}}\frac{l_{ij}+1}{H_{ij}}$ (reduced) is exactly the next break $p_{k+1}=q_{k+1}$, both for column regularization and Mullineux transpose operation. Suppose $\frac{a_{k+1}}{b_{k+1}}\neq1$ and we prove the second claim via the following steps.

Step 1.

Since $$\frac{b_{k+1}-a_{k+1}}{a_{k+1}}=\max_{(i,j)\in\lambda_{k+1}}\frac{a_{ij}}{l_{ij}+1},$$
we know when doing $\Crabk$ to $\lambda_{k+1}$, we can use the exactly same procedure as above and find out ${A}={A^\prime}$ is a removable box and $\overline{B}=\overline{B^\prime}$ is an addable box. This is saying on any ladders which are not full, we are sliding those removable boxes to addable boxes, which indicates $\lambda_{k+1}^{\Crabk}$ is a partition. Moreover, since $\lambda_{k+1}^{\Crabk}\neq\lambda_{k+1}$, we know $\widetilde{\lambda_{k+2}}=\lambda_{k+1}^{\Crabk}$ and $q_k<q_{k+1}=\frac{a_{k+1}}{b_{k+1}}\leq p_{k+1}$.

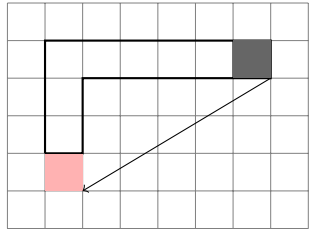
\begin{center}
\begin{tikzpicture}
\draw[step=0.5cm,gray,very thin] (0.5,0) grid (4.5,3);
\draw[thick] (1,1) -- (1.5,1) -- (1.5,2) -- (4,2) -- (4,2.5) -- (1,2.5) -- (1,1);
\draw (3.5,2) -- (3.5,2.5);
\fill[red!30!white] (1,0.5) rectangle (1.5,0.99);
\fill[black!60!white] (3.5,2) rectangle (4,2.5);
\draw [->] (4,2) -- (1.5,0.5);
\end{tikzpicture}
\end{center}

\begin{Remark}
\label{ladderup}
For any ladder $L_c$ which are not full in $\lambda_{k+1}$, the boxes in $L_c\cap \lambda_{k+1}$ always lies above the boxes in $L_c\setminus\lambda_{k+1}$.
\end{Remark}

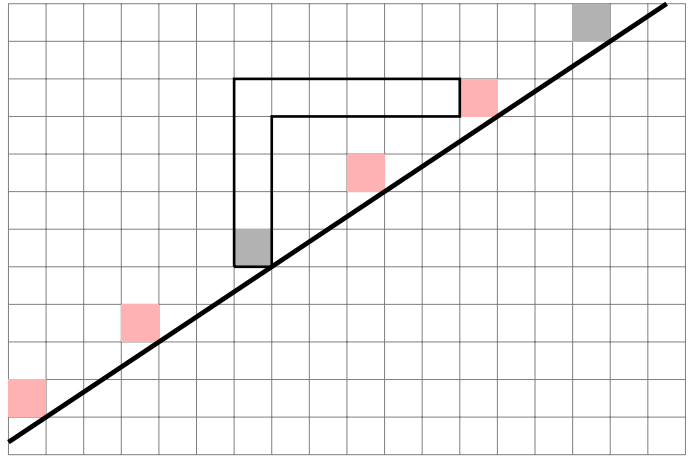
\begin{center}
\begin{tikzpicture}
\draw[step=0.5cm,gray,very thin] (0,0.5) grid (9,6.5);
\fill[red!30!white] (0,1) rectangle (0.5,1.5);
\fill[red!30!white] (1.5,2) rectangle (2,2.5);
\fill[black!30!white] (3,3) rectangle (3.5,3.5);
\fill[red!30!white] (4.5,4) rectangle (5,4.5);
\fill[red!30!white] (6,5) rectangle (6.5,5.5);
\fill[black!30!white] (7.5,6) rectangle (8,6.5);
\draw[line width=1.8pt] (0,2/3) -- (8.75,6.5);
\draw[line width=1pt] (3,3) -- (3,5.5) --(6,5.5) -- (6,5) --(3.5,5) -- (3.5,3) -- (3,3);
\end{tikzpicture}
\end{center}

If there is a box in $L_c\setminus\lambda_{k+1}$ that lies above a box in $L_c\cap \lambda_{k+1}$, then there is a hook of length divisible by $b_{k+1}$ whose endpoint of leg and the box directly right to the endpoint of arm are on the same ladder, as shown in the picture. Suppose this hook has length $H_{i,j}=tb_{k+1}$, then we have $a_{ij}=t(b_{k+1}-a_{k+1})-1$ and $l_{ij}=ta_{k+1}$. Therefore, 
$$\frac{l_{ij}}{H_{ij}}=\frac{a_{k+1}}{b_{k+1}}\leq\max_{(u,v)\in\lambda_{k+1}}\frac{l_{uv}}{H_{uv}}=q_k<q_{k+1}=\frac{a_{k+1}}{b_{k+1}}$$
which is a contradiction.

Step 2.
\begin{Claim}
$$q_{k+1}=\frac{a_{k+1}}{b_{k+1}}=\max_{(i,j)\in\widetilde{\lambda_{k+2}}}\frac{l_{ij}}{H_{ij}}.$$
\end{Claim}
Denote all the sliding boxes in doing $\Crabk$ to $\lambda_{k+1}$ by ${A_1},...,{A_l}$, and each ${A_j}$ slides to $\overline{A_j}$. Let $\tilde{\lambda}\coloneqq \lambda_{k+1}\setminus \{{A_1},...,{A_l}\}=\widetilde{\lambda_{k+2}}\setminus \{\overline{A_1},...,\overline{A_l}\}$.

By construction, we know that $\exists (i_0,j_0)\in\widetilde{\lambda_{k+2}}$, we have $$\frac{l_{i_0,j_0}}{H_{i_0,j_0}}=\frac{a_{k+1}}{b_{k+1}}\leq \max_{(i,j)\in\widetilde{\lambda_{k+2}}}\frac{l_{ij}}{H_{ij}}.$$ 

In addition, for $(i,j)\in\tilde{\lambda}\subset\widetilde{\lambda_{k+2}}$, we have:
$$\frac{a_{ij}(\widetilde{\lambda_{k+2}})+1}{l_{ij}(\widetilde{\lambda_{k+2}})}\geq \frac{a_{ij}({\lambda_{k+1}})}{l_{ij}(\lambda_{k+1})+1}$$
and hence
$$\frac{l_{ij}(\widetilde{\lambda_{k+2}})}{H_{ij}(\widetilde{\lambda_{k+2}})}\leq\frac{l_{ij}(\lambda_{k+1})+1}{H_{ij}(\lambda_{k+1})}$$

$$\max_{(i,j)\in\tilde{\lambda}}\frac{l_{ij}(\widetilde{\lambda_{k+2}})}{H_{ij}(\widetilde{\lambda_{k+2}})}\leq \min_{(i,j)\in\tilde{\lambda}}\frac{l_{ij}(\lambda_{k+1})+1}{H_{ij}(\lambda_{k+1})}   =\min_{(i,j)\in{\lambda_{k+1}}}\frac{l_{ij}+1}{H_{ij}}=q_{k+1}=\frac{a_{k+1}}{b_{k+1}}.$$
The first equality is due to the fact that $\forall{(i,j)\in\{A_1,...,A_l\}}$, there is $\frac{l_{ij}(\lambda_{k+1})+1}{H_{ij}(\lambda_{k+1})}=1$.

Then 
$$\max_{(i,j)\in\widetilde{\lambda_{k+2}}}\frac{l_{ij}}{H_{ij}}=\max\{\max_{(i,j)\in\tilde{\lambda}}\frac{l_{ij}(\widetilde{\lambda_{k+2}})}{H_{ij}(\widetilde{\lambda_{k+2}})}, \max_{(i,j)\in\{\overline{A_1},...,\overline{A_l}\}}\frac{l_{ij}(\widetilde{\lambda_{k+2}})}{H_{ij}(\widetilde{\lambda_{k+2}})}=0\}=\frac{a_{k+1}}{b_{k+1}}.$$

Step 3.
\begin{Claim}
$\tilde{\lambda}$ is a $b_{k+1}\con$core.
\end{Claim}

If $\exists (\tilde{i_0},\tilde{j_0})\in\tilde{\lambda}$, such that $H_{\tilde{i_0},\tilde{j_0}}=\tilde{k_0}b_{k+1}$, then
$$\frac{l_{\tilde{i_0},\tilde{j_0}}}{\tilde{k_0}b_{k+1}}\leq\max_{(i,j)\in\tilde{\lambda}}\frac{l_{ij}}{H_{ij}}<\max_{(i,j)\in\widetilde{\lambda_{k+2}}}\frac{l_{ij}}{H_{ij}} =\frac{a_{k+1}}{b_{k+1}}=\min_{(i,j)\in\lambda_{k+1}}\frac{l_{ij}+1}{H_{ij}}<\min_{(i,j)\in\tilde{\lambda}}\frac{l_{ij}+1}{H_{ij}}\leq \frac{l_{\tilde{i_0},\tilde{j_0}}+1}{\tilde{k_0}b_{k+1}}.$$

Here the first "$<$" is because the arms in $\tilde{\lambda}$ are the same as the corresponding one in $\widetilde{\lambda_{k+2}}$, but legs may be the same or decrease by 1. The second "$<$" is because the legs in $\tilde{\lambda}$ are the same as the corresponding one in $\lambda_{k+1}$, but arms may be the same or decrease by 1. They are strict inequalities since we already removed those sliding boxes to obtain $\tilde{\lambda}$.

The above inequalities simplify to $l_{\tilde{i_0},\tilde{j_0}}<\tilde{k_0}a_{k+1}<l_{\tilde{i_0},\tilde{j_0}}+1$, which is a contradiction.

Step 4.
\begin{Claim}
${A_1},...,{A_l}$ lies on the same ladder.
\end{Claim}

First of all, since $\gcd(a_{k+1},b_{k+1})=1$, the integer boxes on ladder $L_c: \;y+\frac{b_{k+1}}{a_{k+1}}x=\frac{c}{a_{k+1}}$ have the same residue because $y-x=\frac{c-b_{k+1}x}{a_{k+1}}$ and consecutive integer boxes on it has $x\con$coordinates differ by $a_{k+1}$. Hence we'll call the residue of a ladder to be the residue of any integer box on it.

Without loss of generality, say ${A_1}\in L_{c_0}$. We suppose there exists some ${A_j}\in L_{c_1}$ where $c_1\neq c_0$.

Firstly, if $L_{c_0}$ and $L_{c_1}$ have the same residue, all possible integer boxes in $L_{c_0}$ and $L_{c_1}$ will lie on a grid with rectangles of size $a\times (b-a)$. From Remark \ref{ladderup}, we know in $L_{c_0}^+\cap\lambda_{k+1}$ are removable boxes and appear on top of addable boxes $L_{c_0}^+\setminus\lambda_{k+1}$, as shown in the picture where the black line is $L_{c_0}$ and black boxes are in $\lambda$ and red ones are not. 

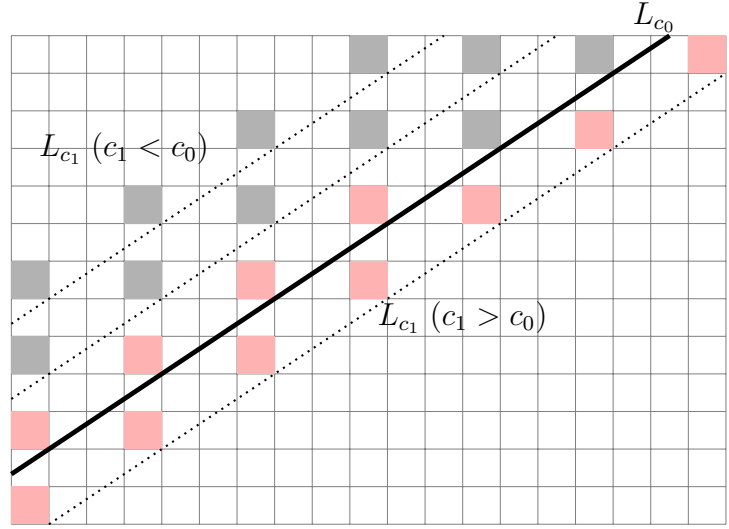
\begin{center}
\begin{tikzpicture}
\draw[step=0.5cm,gray,very thin] (0,0) grid (9.5,6.5);
\fill[red!30!white] (0,0) rectangle (0.5,0.5);
\fill[red!30!white] (1.5,1) rectangle (2,1.5);
\fill[red!30!white] (3,2) rectangle (3.5,2.5);
\fill[red!30!white] (4.5,3) rectangle (5,3.5);
\fill[red!30!white] (6,4) rectangle (6.5,4.5);
\fill[red!30!white] (7.5,5) rectangle (8,5.5);
\fill[red!30!white] (9,6) rectangle (9.5,6.5);
\fill[red!30!white] (0,1) rectangle (0.5,1.5);
\fill[red!30!white] (1.5,2) rectangle (2,2.5);
\fill[red!30!white] (3,3) rectangle (3.5,3.5);
\fill[red!30!white] (4.5,4) rectangle (5,4.5);
\fill[black!30!white] (6,5) rectangle (6.5,5.5);
\fill[black!30!white] (7.5,6) rectangle (8,6.5);
\fill[black!30!white] (0,2) rectangle (0.5,2.5);
\fill[black!30!white] (1.5,3) rectangle (2,3.5);
\fill[black!30!white] (3,4) rectangle (3.5,4.5);
\fill[black!30!white] (4.5,5) rectangle (5,5.5);
\fill[black!30!white] (6,6) rectangle (6.5,6.5);
\fill[black!30!white] (0,3) rectangle (0.5,3.5);
\fill[black!30!white] (1.5,4) rectangle (2,4.5);
\fill[black!30!white] (3,5) rectangle (3.5,5.5);
\fill[black!30!white] (4.5,6) rectangle (5,6.5);
\draw[line width=1.8pt] (0,2/3) -- (8.75,6.5);
\draw[dotted,line width=.8pt] (0,5/3) -- (7.25,6.5);
\draw[dotted,line width=.8pt] (0.5,0) -- (9.5,6);
\draw[dotted,line width=.8pt] (0,8/3) -- (5.75,6.5);
\foreach \x/\y/\m in {+8.55/+6.75/$L_{c_0}$,+6/+2.75/$L_{c_1}\;(c_1>c_0)$,+1.5/+5/$L_{c_1}\;(c_1<c_0)$} 
    \node at (\x,\y) {\m};

\end{tikzpicture}
\end{center}

Using the definition of addable and removable boxes, we know any box on the grid and southeast to $L_{c_0}^+$ is not in $\lambda_{k+1}$ and any box on the grid and northwest to $L_{c_0}^+$ is in $\lambda_{k+1}$. This indicates when $c_1<c_0$, $L_{c_1}^+$ will be full, when $c_1>c_0$, $L_{c_1}^+$ will be empty, i.e. $L_{c_1}^+\cap\lambda_{k+1}=\emptyset$. In either case, there will be no ${A_j}$ in $L_{c_1}$.

Then we are left with the case when $L_{c_0}$ and $L_{c_1}$ have different residues. If the distance $\dist(L_{c_0},L_{c_1})>\frac{a(b-a)}{\sqrt{a^2+(b-a)^2}}$, then the same reasoning as above will show that $L_{c_1}$ is either empty or full, which is not possible. If $\dist(L_{c_0},L_{c_1})<\frac{a(b-a)}{\sqrt{a^2+(b-a)^2}}$, without loss of generality, we also assume $c_1>c_0$ find ${E}\in L_{c_1}\cap\lambda_{k+1},$ and $\overline{F}\in L_{c_0}\setminus\lambda_{k+1}$ where ${E}$ is above $\overline{F}$. This is always possible since $c_1>c_0$ and both ladders contain boxes both in and not in $\lambda_{k+1}$. From Step 1, ${E}$ is removable and $\overline{F}$ is addable. Now the hook in $\lambda_{k+1}$ corresponding to ${E}$ and $\overline{F}$ will break the inequality $\frac{b_{k+1}-a_{k+1}}{a_{k+1}}\geq\frac{a_{ij}}{l_{ij}+1}$.

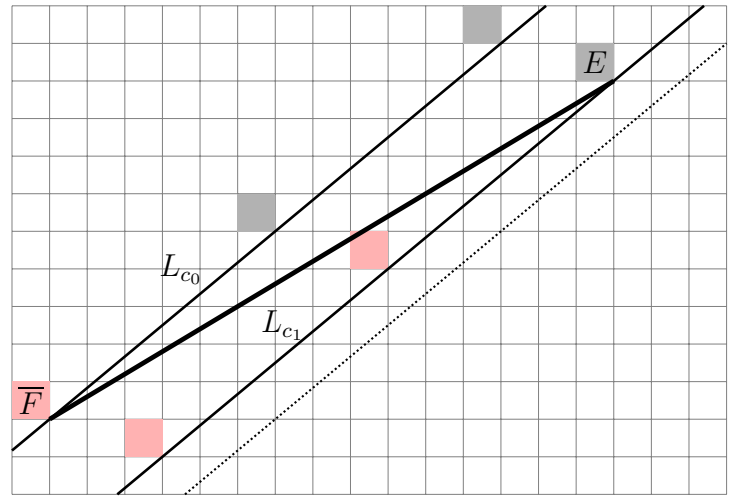
\begin{center}
\begin{tikzpicture}
\draw[step=0.5cm,gray,very thin] (0,0.5) grid (9.5,7);
\fill[red!30!white] (0,1.5) rectangle (0.5,2);
\fill[black!30!white] (3,4) rectangle (3.5,4.5);
\fill[black!30!white] (6,6.5) rectangle (6.5,7);

\fill[red!30!white] (1.5,1) rectangle (2,1.5);
\fill[red!30!white] (4.5,3.5) rectangle (5,4);
\fill[black!30!white] (7.5,6) rectangle (8,6.5);

\draw[line width=1pt] (0,13/12) -- (7.1,7);
\draw[densely dotted, line width=0.8pt] (2.3,0.5) -- (9.5,6.5);
\draw[line width=1pt] (1.4,0.5) -- (9.2,7);
\foreach \x/\y/\m in {+0.25/+1.75/$\overline{F}$,7.75/6.25/$E$,+3.6/+2.75/$L_{c_1}$,+2.25/+3.5/$L_{c_0}$} 
    \node at (\x,\y) {\m};
%\draw[line width=1.8pt] (0,1.2) -- (29/3,7);
\draw[line width=1.8pt] (0.5,1.5) -- (8,6);
\end{tikzpicture}
\end{center}

Now we arrived at the conclusion that ${A_1},...,{A_l}$ are on the same ladder. We denote this special ladder $L^{k+1}_*$ from now on.

Step 5.

Let's now prove $\lambda_{k+1}$ is $b_{k+1}\con$regular. If not, there is a box $(i_1,j_1)\in\lambda_{k+1}$ with the corresponding hook being a strip of length $b_{k+1}$ in the rim. Then we have $$\frac{l_{i_1,j_1}}{H_{i_1,j_1}}=\frac{b_{k+1}-1}{b_{k+1}}\leq\max_{(i,j)\in\lambda_{k+1}}\frac{l_{ij}}{H_{ij}}=p_k=q_k<\frac{a_{k+1}}{b_{k+1}}.$$
This contradicts our assumption that $\frac{a_{k+1}}{b_{k+1}}\neq1$ and hence $\lambda_{k+1}$ is $b_{k+1}\con$regular.

We'll now construct $\lambda_{k+1}^{\MkT}$ by decomposing $\lambda_{k+1}$ into a good box sequence and build them back using the same sequence as co-good decomposition sequence as stated in Theorem \ref{goodco-goodmt}.

Assume the integer boxes on $L^{k+1,+}_*$ is labelled by $A_1,...,A_m$ in order from northeast to southwest. Then from Remark \ref{ladderup}, the first $l$ are exactly our ${A_1},...,{A_l}$, and the rest $\overline{A_{l+1}},...,\overline{A_m}$ are addable boxes. Definition \ref{good} of good box indicates $A_l$ is a good box of $\lambda_{k+1}$ and ${A_{l-1}}$ is a good box of $\lambda_{k+1}\setminus {A_l}$ and it continues. So we get a good decomposition sequence of $\lambda_{k+1}$: ${A_l},...,{A_1},{G_1},...,G_{n-l}$, where ${G_1},...,{G_{n-l}}$ is a good decomposition sequence for the $b_{k+1}\con$core $\tilde{\lambda}$.

Since $\tilde{\lambda}$ is a $b_{k+1}\con$core, ${G_1},...,{G_{n-l}}$ is also a co-good decomposition sequence by Proposition \ref{MTcore}. Afterwards, we put ${A_1}$ to $\overline{A_m}$, since by definition, $\overline{A_m}$ is co-good in $\tilde{\lambda}\cup\overline{A_m}$. Then we put ${A_j}$ to position $\overline{{A_{m+1-j}}}$ in order since $\overline{A_j}$ is co-good in $\tilde{\lambda}\cup\overline{A_m}\cup\cdots\cup\overline{A_j}$. 

Hence $\lambda_{k+1}^{\MkT}=\tilde{\lambda}\cup\overline{A_m}\cup\cdots\cup\overline{A_{m+1-l}}$, and this is exactly sliding $ {A_1},..., {A_l}$ in order to the bottom of $L^{k+1,+}_*$, so $\lambda_{k+1}^{\MkT}=\lambda_{k+1}^{\Crabk}$.
\end{proof}

\begin{Corollary}
\label{endcolumn}
The partition in the last interval $[\frac{n-1}{n},1]$ is exactly the column $(1^n)$.
\end{Corollary}

\begin{proof}
In characteristic $2$, sign representation is exactly the trivial representation, hence $\lambda^{\Mtwo}=\lambda$ for every $2\con$regular partition $\lambda$. Also, $$(\lambda^{\Tr})^{\Crbab}=(\lambda^{\Crab})^{\Tr},$$which is direct by Definition \ref{colreg}. Hence the sequence of partitions are symmetric via transpose at $\frac{1}{2}$. So $\lambda_{[\frac{n-1}{n},1]}=\widetilde{\lambda_{[\frac{n-1}{n},1]}}=(n)^{\Tr}=(1^n)$.
\end{proof}
%\begin{Corollary}
%$\lambda_k\geq\lambda_{k+1}$ under the lexicographic or dominant order for all $k$.
%\end{Corollary}
%\begin{proof}
%This is immediate from the definition of column regularization.
%\end{proof}

%\begin{Lemma}
%\label{mostlysymmetric}
%Let $I_0=\left[\con,\frac{1}{2}\right]$ be the interval with endpoint $\frac{1}{2}$, then the partition $B_{I_0}((n))=\lambda_{I_0}$ satisfies it is the maximal among all $2\con$regular partition of $n$.
%\end{Lemma}
%\begin{proof}
%In the proof of Theorem \ref{main}, there is a special ladder $L_*$ such that all sliding boxes lies on the top places of this ladder. Here the fraction is $\frac{1}{2}$, hence the ladders we are considering with all have slope $-1$. Hence $B_{I_0}((n))=\lambda_{I_0}$ can be constructed via the following process. 

%First fill the first ladder $L_2$ with one box, namely $(1,1)$, then continue with ladder $L_3,....$ until we reach $L_r$ where $\sum_{j=1}^{r-1}j<n\leq\sum_{j=1}^{r}j$. Then we fill the top $n-\sum_{j=1}^{r-1}j$ boxes in the last ladder $L_{r+1}=L_*$. And the resulting partition is exactly $B_{I_0}((n))=\lambda_{I_0}$. The construction directly indicates that it is the smallest $2\con$regular partition of $n$, i.e. the partitions with distinct row lengths.
%\end{proof}
\section{Detailed Descriptions of the Sequence of Partitions}
\label{secquotient}
From the constructions in Section \ref{mainsection}, we know the sequence $\lambda_I$ is a series of decreasing partitions starting from the row $(n)$ and ending at the column $(1^n)$. In this section, we provide more details to the partitions $\lambda_k$'s.

\begin{Corollary}
The $b_{k}\con$quotient of  $\lambda_k$ is $\Quot_{b_{k}}(\lambda_k)=(\emptyset,...,\emptyset,(h_2^{h_1}),\emptyset,...,\emptyset)$, where the only nonempty entry is a rectangle. Here $h_1=|L_*^{k,+}\cap\lambda_k|$ and $h_2=|L_*^{k,+}\setminus\lambda_k|$ and the rectangle appears at the $(j_k+1)\con$th entry where $j_k$ is the residue of $L_*^k$.

Moreover, the $b_{k}\con$quotient of  $\lambda_{k+1}$ is $\Quot_{b_{k}}(\lambda_{k+1})=(\emptyset,...,\emptyset,(h_1^{h_2}),\emptyset,...,\emptyset)$, where the rectangle appears at the $j_k\con$th entry.

In addition, the $b_k\con$core of $\lambda_{k+1}$ and $\lambda_{k}$ are the same.
\end{Corollary}

\begin{proof}
From \cite{ford1997proof}, $\Core_{b_k}(\lambda_{k+1})=\Core_{b_k}(\lambda_{k})$ since $\lambda_{k+1}=\lambda_{k}^{\MkkT}$.

Denote $\Quot_{b_{k}}(\lambda_{k+1})=(\nu_{k,0},...,\nu_{k,b_k-1})$ from Definition \ref{haimanquotient}. $\nu_{k,s}$ is exactly the exploded boxes $(i,j)$ with corresponding hook divisible by $b_k$ and the residue of the box at the end of the arm has residue $s$. Let $H_{ij}=tb_k$, and since $\frac{a_k}{b_k}=\min_{(i,j)\in\lambda_k}\frac{l_{ij}+1}{H_{ij}}>\max_{(i,j)\in\lambda_{k}}\frac{l_{ij}}{H_{ij}}$, we have:
$$\frac{l_{ij}}{H_{ij}}<\frac{a_k}{b_k}\leq\frac{l_{ij}+1}{H_{ij}},$$
hence we get $l_{ij}=ta_k-1$ and $a_{ij}=t(b_k-a_k)$.

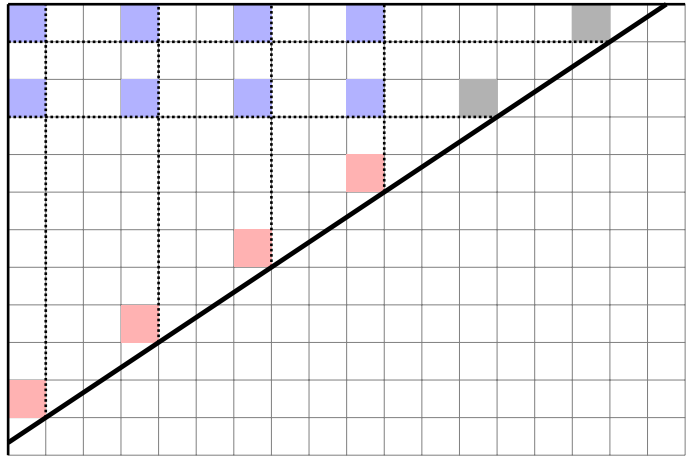
\begin{center}
\begin{tikzpicture}
\draw[step=0.5cm,gray,very thin] (0,0.5) grid (9,6.5);
\fill[red!30!white] (0,1) rectangle (0.5,1.5);
\fill[red!30!white] (1.5,2) rectangle (2,2.5);
\fill[red!30!white] (3,3) rectangle (3.5,3.5);
\fill[red!30!white] (4.5,4) rectangle (5,4.5);
\fill[black!30!white] (6,5) rectangle (6.5,5.5);
\fill[black!30!white] (7.5,6) rectangle (8,6.5);
\fill[blue!30!white] (0,5) rectangle (0.5,5.5);
\fill[blue!30!white] (1.5,5) rectangle (2,5.5);
\fill[blue!30!white] (3,5) rectangle (3.5,5.5);
\fill[blue!30!white] (4.5,5) rectangle (5,5.5);
\fill[blue!30!white] (0,6) rectangle (0.5,6.5);
\fill[blue!30!white] (1.5,6) rectangle (2,6.5);
\fill[blue!30!white] (3,6) rectangle (3.5,6.5);
\fill[blue!30!white] (4.5,6) rectangle (5,6.5);
\draw[densely dotted,line width=1pt] (0.5,1) -- (0.5,6.5);
\draw[densely dotted,line width=1pt] (2,2) -- (2,6.5);
\draw[densely dotted,line width=1pt] (3.5,3) -- (3.5,6.5);
\draw[densely dotted,line width=1pt] (5,4) -- (5,6.5);
\draw[densely dotted,line width=1pt] (0,5) -- (6.5,5);
\draw[densely dotted,line width=1pt] (0,6) -- (8,6);
\draw[line width=1pt] (0,0.5) -- (0,6.5);
\draw[line width=1pt] (0,6.5) -- (9,6.5);
\draw[line width=1.8pt] (0,2/3) -- (8.75,6.5);
\end{tikzpicture}
\end{center}

This is exactly saying the endpoint of the arm of $(i,j)$ and the box directly underneath the endpoint of the leg of $(i,j)$ are on $L_*^{k,+}$. Hence $\nu_{k,s_0}=(h_2^{h_1})$ where $s_0$ is the residue of $L_*^k$, $h_1=|L_*^{k,+}\cap\lambda_k|$ and $h_2=|L_*^{k,+}\setminus\lambda_k|$ and all other entries in the quotient are $\emptyset$.

Now consider $\lambda_{k+1}=\lambda_{k}^{\MkkT}$ and denote the $b_{k}\con$quotient of it by $(\xi_{k,0},...,\xi_{k,b_k-1})$. From Theorem \ref{main}, we have
$$\frac{a_k}{b_k}=\max_{(i,j)\in\lambda_{k+1}}\frac{l_{ij}}{H_{ij}}<\min_{(i,j)\in\lambda_{k+1}}\frac{l_{ij}+1}{H_{ij}}.$$

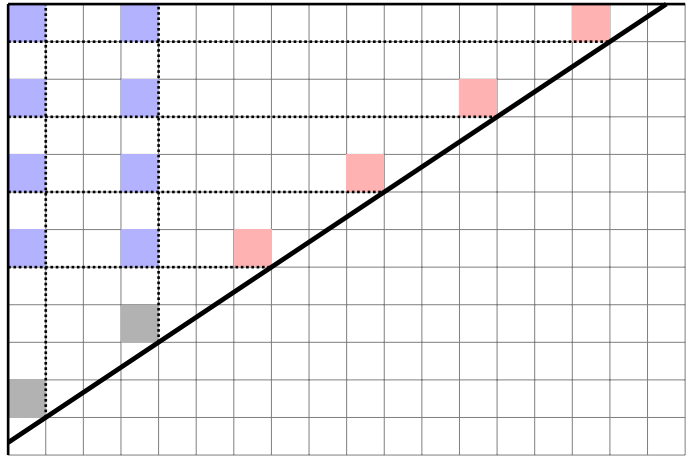
\begin{center}
\begin{tikzpicture}
\draw[step=0.5cm,gray,very thin] (0,0.5) grid (9,6.5);
\fill[black!30!white] (0,1) rectangle (0.5,1.5);
\fill[black!30!white] (1.5,2) rectangle (2,2.5);
\fill[red!30!white] (3,3) rectangle (3.5,3.5);
\fill[red!30!white] (4.5,4) rectangle (5,4.5);
\fill[red!30!white] (6,5) rectangle (6.5,5.5);
\fill[red!30!white] (7.5,6) rectangle (8,6.5);
\fill[blue!30!white] (0,5) rectangle (0.5,5.5);
\fill[blue!30!white] (1.5,5) rectangle (2,5.5);
\fill[blue!30!white] (0,6) rectangle (0.5,6.5);
\fill[blue!30!white] (1.5,6) rectangle (2,6.5);
\fill[blue!30!white] (0,4) rectangle (0.5,4.5);
\fill[blue!30!white] (0,3) rectangle (0.5,3.5);
\fill[blue!30!white] (1.5,4) rectangle (2,4.5);
\fill[blue!30!white] (1.5,3) rectangle (2,3.5);
\draw[densely dotted,line width=1pt] (0,5) -- (6.5,5);
\draw[densely dotted,line width=1pt] (0,6) -- (8,6);
\draw[densely dotted,line width=1pt] (0,4) -- (5,4);
\draw[densely dotted,line width=1pt] (0,3) -- (3.5,3);
\draw[densely dotted,line width=1pt] (0.5,1) -- (0.5,6.5);
\draw[densely dotted,line width=1pt] (2,2) -- (2,6.5);
\draw[line width=1pt] (0,0.5) -- (0,6.5);
\draw[line width=1pt] (0,6.5) -- (9,6.5);
\draw[line width=1.8pt] (0,2/3) -- (8.75,6.5);
\end{tikzpicture}
\end{center}

When $(i^\prime,j^\prime)\in\lambda_{k+1}$ satisfies $H_{ij}=t^\prime b_k$, we have the following inequality:
$$\frac{l_{i^\prime j^\prime}}{t^\prime b_k}\leq\frac{a_k}{b_k}<\frac{l_{i^\prime j^\prime}+1}{t^\prime b_k},$$
and this simplifies to $l_{i^\prime j^\prime}=t^\prime a_k$ and $a_{i^\prime j^\prime}=t^\prime(b_k-a_k)-1$. This indicates the endpoint of the leg of $(i^\prime,j^\prime)$ and the box directly right to the endpoint of the arm of $(i^\prime,j^\prime)$ are both on $L_*^{k,+}$. Therefore, $\xi_{k,s_0-1}=(h_1^{h_2})$ and all other entries in the $b_{k}\con$quotient of $\lambda_{k+1}$ are empty partitions.

\end{proof}

\section{Uniqueness of Monotonicity}
\label{secunique}
From the proof of Theorem \ref{main}, we have already seen that $\lambda_I=B_I((n))$ is $b\con$regular if the right endpoint of $I$ has denominator $b$. Now we will show that $(n)$ is the unique partition that always stays regular under the series of combinatorial wall-crossings.

\begin{Lemma}
\label{preuniq}
Given a $b$-regular partition $\lambda$, if $b\mid \phi(\lambda)$ where $\phi(\lambda)=|\lambda|-|\lambda^{\I}|$, then $$j_1=|\lambda|-|\lambda^{\J}|\le \lambda_1-1.$$
\end{Lemma}

\begin{proof}
$j_1=|\lambda|-|\lambda^{\J}|$, where $|\lambda^{\J}|=|\lambda|-|\lambda^{\I}|+k$ since $\delta=1$ as in Definition \ref{defJ}. But since the $b\con$rim is a subset of the rim, $|\lambda^{\J}|\geq|\lambda|-(k+\lambda_1-1)+k=|\lambda|-\lambda_1+1$. Substituting back, we immediately obtain the required inequality.
\end{proof}

\begin{Lemma}
\label{pre2lem}
Given a $b\con$regular partition $\lambda$, if $b\mid H_{11}$, then $j_1\le \lambda_1-1$.
\end{Lemma}

\begin{proof}
The condition $b\mid H_{11}$ implies that the number of boxes in the rim is divisible by $b$. If $b\mid \phi(\lambda)$ then we are done from Lemma \ref{preuniq}. If $b\nmid \phi(\lambda)$ then since the $b\con$rim is not divisible by $b$, it is strictly smaller than the rim so $|\lambda^{\J}|>|\lambda|-(k+\lambda_1-1)+(k-1)$. Substituting to the formula giving $j_1$ we get the desired inequality.
\end{proof}

\begin{Proposition}
\label{pre3prop}
Begin with any partition $\lambda$ of $n$ and perform the wall-crossing transformation, as long as the partition stays regular, it will hold that $$\frac{l_{11}}{H_{11}}>\frac{a}{b}$$ in the partition exactly after we cross the wall $\frac{a}{b}\in F_n$.
\end{Proposition}

\begin{proof}
We prove this property by induction. For any partition $\lambda$ of $n$, $l_{11}\geq 1$ and $H_{11}\leq n$, hence in $\lambda$ we have $\frac{l_{11}}{H_{11}}>\frac{1}{n}$.

Assume that we have two consecutive terms of the Farey sequence $\frac{a}{b}<\frac{c}{d}$. We know by induction that the partition immediately before we cross $\frac{c}{d}$ satisfies $\frac{l_{11}}{H_{11}}>\frac{a}{b}$. Now since two consecutive elements in the Farey sequence give two consecutive slopes in an $n\times n$ square in the cartesian grid, it must hold that $\frac{l_{11}}{H_{11}}\geq\frac{c}{d}$. If the inequality is strict then we are done. If this is an equality, then $d\mid H_{11}$, so by Lemma \ref{pre2lem}, $a'_{11}+1=j_1\le \lambda_1-1=a_{11}$. Also the fact that Mullineux transpose respects the partial order \cite[Corollary 4.4]{bessenrodt1999properties} implies that the number of rows of $\lambda^{\MdT}$ is at least as much as the number of rows of $\lambda$. Thus after crossing $\frac{c}{d}$,
$$\frac{l'_{11}}{H'_{11}}-\frac{l_{11}}{H_{11}}=\frac{l'_{11}(a_{11}+1)-l_{11}(a'_{11}+1)}{H'_{11}H_{11}}>0$$ which implies the desired strict inequality. 
\end{proof}

\begin{Theorem}
\label{monotone}
The row partition $(n)$ is the unique partition of $n$ that stays regular after any step of the combinatorial wall-crossing transformation. For any other partition $\lambda$ of $n$ there exist an interval $I=\left[\frac{a_i}{b_i}, \frac{a_{i+1}}{b_{i+1}}\right]$ in $[0,1]$ defined by the Farey sequence such that $B_{I}(\lambda)$ is not $b_{i+1}\con$regular. 
\end{Theorem}
\begin{proof}
Suppose there is a partition $\lambda\neq (n)$ of $n$ that stays regular after any step of the combinatorial wall-crossing transformation, consider $\mu=B_{\left[\frac{n-1}{n},1\right]}(\lambda)$. By Proposition \ref{pre3prop}, we have that in $\mu$, $\frac{l_{11}}{H_{11}}>\frac{n-1}{n}$, which is equivalent to $\frac{a_{11}+1}{l_{11}}<\frac{1}{n-1}$. But this inequality happens if and only if $\mu=(1^n)$. But by Remark \ref{Bbij} and Corollary \ref{endcolumn}, we know $\lambda=(n)$, which is a contradiction.

\end{proof}

\appendix

\section{Example of $n=5$ and a General Conjecture}
\label{appendixa}
In this appendix, we'll denote $B_I$ and $D_I$ in Definition \ref{CWC} as $B_I^2$ and $D_I^2$ respectively and define another series of transformations as follows.

\begin{Definition}
\label{firstalgo}
We define a collection of maps $B^1_I:\mathcal{P}_n\rightarrow\mathcal{P}_n$  where $I$ are intervals with endpoints being consecutive terms in $F_n$ as follows. First of all $B_{\left[0,\frac{1}{n}\right]}^1(\lambda)=\lambda$ is the identity map. Inductively, suppose we defined  $B_{I}^1$ where $I=\left[\frac{a_{i-1}}{b_{i-1}}, \frac{a_i}{b_i}\right]$. Then for the adjacent interval $I^\prime=\left[\frac{a_i}{b_i}, \frac{a_{i+1}}{b_{i+1}}\right]$, we define $B_{I^\prime}^1(\lambda)=\mu\cup b_i\nu^{\Tr}$ and $D^1_{I}(\lambda)=b_i\cdot|\nu|$. Here, $\mu$ and $\nu$ come from the unique decomposition $B_{I}^1(\lambda)=\mu\cup b_i\nu$ where $\mu$ has no parts divisible by $b_i$ and $b_i\nu$ is multiplying each part of $\nu$ by $b_i$.

 \end{Definition}

Farey sequence of $n=5$ is 
$$ \frac{1}{5}, \frac{1}{4}, \frac{1}{3}, \frac{2}{5}, \frac{1}{2}, \frac{3}{5}, \frac{2}{3}, \frac{3}{4}, \frac{4}{5}.$$

Now we perform the two algorithms as in Definition \ref{CWC} and \ref{firstalgo} to all partitions of $5$ and calculate the corresponding number function $D_I^1$ and $D_I^2$ respectively indicating the row and column irregular sizes respectively. The results are given in the following tables:
\renewcommand\arraystretch{1.5}
\begin{table}[!htb]
\centering
\caption {The First Algorithm, n=5}
\begin{tabular}{|c|c|c|c|c|c|c|c|c|c|c|}

\hline
Interval & $[0,\frac{1}{5}]$ & $[\frac{1}{5},\frac{1}{4}]$ & $[\frac{1}{4},\frac{1}{3}]$ & $[\frac{1}{3},\frac{2}{5}]$ & $[\frac{2}{5},\frac{1}{2}]$ & $[\frac{1}{2},\frac{3}{5}]$ & $[\frac{3}{5},\frac{2}{3}]$  & $[\frac{2}{3},\frac{3}{4}]$ & $[\frac{3}{4},\frac{4}{5}]$ & $[\frac{4}{5},1]$\\
\hline\hline
$B^1_I$ & $(5)$ & $(5)$ & $(5)$ & $(5)$ & $(5)$ & $(5)$ & $(5)$ & $(5)$ & $(5)$ & $(5)$ \\
\hline
$D^1_I$ & 5 & 0 & 0 & 5 & 0 & 5 & 0 & 0 & 5 & -\\
\hline\hline
 $B^1_I$ & $(4,1)$ &  $(4,1)$ &  $(4,1)$ &  $(4,1)$ &  $(4,1)$ &  $(2^2,1)$ & $(2^2,1)$ & $(2^2,1)$ & $(2^2,1)$ & $(2^2,1)$ \\
 \hline
 $D^1_I$ & 0 & 4 & 0 & 0 & 4 & 0 & 0 & 0 & 0 & -\\
 \hline\hline
 $B^1_I$ & $(3,2)$ & $(3,2)$ & $(3,2)$ & $(3,2)$ & $(3,2)$ & $(3,2)$ & $(3,2)$ & $(3,2)$ & $(3,2)$ & $(3,2)$ \\
 \hline
 $D^1_I$ & 0 & 0 & 3 & 0 & 2 & 0 & 3 & 0 & 0 & -\\
 \hline\hline
$B^1_I$ & $(3,1^2)$ & $(3,1^2)$ & $(3,1^2)$ & $(3,1^2)$ & $(3,1^2)$ & $(3,1^2)$ & $(3,1^2)$ & $(3,1^2)$ & $(3,1^2)$ & $(3,1^2)$ \\
\hline
$D^1_I$ & 0 & 0 & 3 & 0 & 0 & 0 & 3 & 0 & 0 & -\\
 \hline\hline
$B^1_I$ & $(2^2,1)$ & $(2^2,1)$ & $(2^2,1)$ & $(2^2,1)$ & $(2^2,1)$ & $(4,1)$ & $(4,1)$ & $(4,1)$ & $(4,1)$ & $(4,1)$ \\
\hline
$D^1_I$ & 0 & 0 & 0 & 0 & 4 & 0 & 0 & 4 & 0 & -\\
 \hline\hline
$B^1_I$ & $(2,1^3)$ & $(2,1^3)$ & $(2,1^3)$ & $(2,1^3)$ & $(2,1^3)$ & $(2,1^3)$ & $(2,1^3)$ & $(2,1^3)$ & $(2,1^3)$ & $(2,1^3)$\\
\hline
$D^1_I$ & 0 & 0 & 0 & 0 & 2 & 0 & 0 & 0 & 0 & -\\
 \hline\hline
$B^1_I$ & $(1^5)$ & $(1^5)$ & $(1^5)$ & $(1^5)$ & $(1^5)$ & $(1^5)$ & $(1^5)$ & $(1^5)$ & $(1^5)$ & $(1^5)$ \\
\hline
$D^1_I$ & 0 & 0 & 0 & 0 & 0 & 0 & 0 & 0 & 0 & -\\
 \hline\hline
\end{tabular}
\end{table}

\renewcommand\arraystretch{1.5}
\begin{table}[!htb]
\centering
\caption {The Second Algorithm, n=5}
\begin{tabular}{|c|c|c|c|c|c|c|c|c|c|c|}

\hline
Interval & $\left[0,\frac{1}{5}\right]$ & $\left[\frac{1}{5},\frac{1}{4}\right]$ & $\left[\frac{1}{4},\frac{1}{3}\right]$ & $\left[\frac{1}{3},\frac{2}{5}\right]$ & $\left[\frac{2}{5},\frac{1}{2}\right]$ & $\left[\frac{1}{2},\frac{3}{5}\right]$ & $\left[\frac{3}{5},\frac{2}{3}\right]$  & $\left[\frac{2}{3},\frac{3}{4}\right]$ & $\left[\frac{3}{4},\frac{4}{5}\right]$ & $\left[\frac{4}{5},1\right]$\\
\hline\hline
$B_I^2$ & $(1^5)$ & $(5)$ & $(3,2)$ & $(1^5)$ & $(5)$ & $(1^5)$ & $(5)$ & $(2^2,1)$ & $(1^5)$ & $(5)$ \\
\hline
$D_I^2$ & 5 & 0 & 0 & 5 & 0 & 5 & 0 & 0 & 5 & -\\
\hline\hline
 $B_I^2$ & $(2,1^3)$ &  $(1^5)$ &  $(5)$ &  $(2^2,1)$ &  $(2^2,1)$ &  $(5)$ & $(4,1)$ & $(3,2)$ & $(2^2,1)$ & $(2^2,1)$ \\
 \hline
 $D_I^2$ & 0 & 4 & 0 & 0 & 4 & 0 & 0 & 0 & 0 & -\\
 \hline\hline
 $B_I^2$ & $(2^2,1)$ & $(2^2,1)$ & $(1^5)$ & $(4,1)$ & $(3,1^2)$ & $(3,1^2)$ & $(2,1^3)$ & $(5)$ & $(3,2)$ & $(3,2)$ \\
 \hline
 $D_I^2$ & 0 & 0 & 3 & 0 & 2 & 0 & 3 & 0 & 0 & -\\
 \hline\hline
$B_I^2$ & $(3,1^2)$ & $(2,1^3)$ & $(2,1^3)$ & $(5)$ & $(4,1)$ & $(2,1^3)$ & $(1^5)$ & $(4,1)$ & $(4,1)$ & $(3,1^2)$ \\
\hline
$D_I^2$ & 0 & 0 & 3 & 0 & 0 & 0 & 3 & 0 & 0 & -\\
 \hline\hline
$B_I^2$ & $(3,2)$ & $(3,2)$ & $(2^2,1)$ & $(2,1^3)$ & $(1^5)$ & $(3,2)$ & $(3,2)$ & $(1^5)$ & $(5)$ & $(4,1)$ \\
\hline
$D_I^2$ & 0 & 0 & 0 & 0 & 4 & 0 & 0 & 4 & 0 & -\\
 \hline\hline
$B_I^2$ & $(4,1)$ & $(3,1^2)$ & $(3,1^2)$ & $(3,1^2)$ & $(2,1^3)$ & $(4,1)$ & $(3,1^2)$ & $(3,1^2)$ & $(3,1^2)$ & $(2,1^3)$\\
\hline
$D_I^2$ & 0 & 0 & 0 & 0 & 2 & 0 & 0 & 0 & 0 & -\\
 \hline\hline
$B_I^2$ & $(5)$ & $(4,1)$ & $(4,1)$ & $(3,2)$ & $(3,2)$ & $(2^2,1)$ & $(2^2,1)$ & $(2,1^3)$ & $(2,1^3)$ & $(1^5)$ \\
\hline
$D_I^2$ & 0 & 0 & 0 & 0 & 0 & 0 & 0 & 0 & 0 & -\\
 \hline\hline
\end{tabular}
\end{table} 

The following conjecture is due to Bezrukavnikov:
\begin{Conjecture}[Bezrukavnikov]
$D_I^1(\lambda)=D^2_I(\lambda^{\Tr})$ for every partition $\lambda$ of $n$ and interval $I$ with endpoints consecutive entries in Farey sequence of $n$.
\end{Conjecture}
\begin{Remark}
The monotonicity of case $(n)$ in the second algorithm is direct if the conjecture is true since $B_I^1((1^n))=1^n$ for every interval $I$ and $D_I^1((1^n))=0$. Hence $D_I^2((n))=0$, which is exactly saying $B^2_I((n))=\lambda_I$ is always regular corresponding to the denominator of the right endpoint of $I$.
\end{Remark}

%\end{appendices}

\newpage
\bibliographystyle{alpha}
\bibliography{MTCOLREG}

\end{document}